\documentclass{article}[12pt]
\usepackage{color}
\usepackage{mathrsfs}
\usepackage{mathtools}
\usepackage{graphicx}
\usepackage{epstopdf}
\usepackage{stmaryrd}
 \usepackage{subfigure}
\usepackage{float}
\usepackage{amsmath}
    \usepackage{bm}
\usepackage{amsfonts,amssymb}
  \usepackage{dsfont}
\usepackage{amsfonts}
\usepackage{mathrsfs}
  \usepackage{pifont}
\numberwithin{equation}{section}
 \usepackage{amsthm}
 \newtheorem{theorem}{Theorem}[section]
\newtheorem{lemma}[theorem]{Lemma}

\newtheorem{corollary}[theorem]{Corollary}
\newtheorem{remark}[theorem]{Remark}
\newtheorem{assumption}[theorem]{Assumption}
 \usepackage{amsmath}
 \usepackage{multirow}

\begin{document}
\title{Analysis of nonconforming  IFE methods and a new scheme for elliptic interface problems}
\author{
Haifeng Ji\footnotemark[1]\qquad
Feng Wang\footnotemark[2] \qquad
Jinru Chen\footnotemark[3] \qquad
Zhilin Li\footnotemark[4]
}
\footnotetext[1]{School of Science, Nanjing University of Posts and Telecommunications, Nanjing 210023, China  (hfji@njupt.edu.cn; hfji1988@foxmail.com)}
\footnotetext[2]{Jiangsu Key Laboratory for NSLSCS, School of Mathematical Sciences, Nanjing Normal University, Nanjing 210023, China  (fwang@njnu.edu.cn)}
\footnotetext[3]{Jiangsu Key Laboratory for NSLSCS, School of Mathematical Sciences, Nanjing Normal University, Nanjing 210023, China \& School of Mathematical Sciences, Jiangsu Second Normal University, Nanjing 211200, China  (jrchen@njnu.edu.cn)}
\footnotetext[4]{Department of Mathematics, North Carolina State University, Raleigh, NC 27695, USA  (zhilin@math.ncsu.edu)}

\date{}
\maketitle

\begin{abstract}
In this paper, an important discovery has been found for  nonconforming immersed finite element (IFE) methods using the integral values on edges as degrees of freedom for solving elliptic interface problems. We show that those IFE methods without penalties are not guaranteed to converge optimally if the tangential derivative of the exact solution and the jump of the coefficient are not zero on the interface. A nontrivial counter  example  is also provided to  support our theoretical analysis. 
To recover the optimal convergence rates, we develop a new nonconforming IFE method with additional terms locally on interface edges. The new method is parameter-free which removes the limitation of the conventional partially penalized IFE method. We show the IFE basis functions are unisolvent on arbitrary triangles which is not considered in the literature. Furthermore, different from multipoint Taylor expansions, we derive the optimal approximation capabilities of both the Crouzeix-Raviart and the rotated-$Q_1$ IFE spaces via a unified approach which can handle the case of variable coefficients easily.
Finally, optimal error estimates  in  both $H^1$- and $L^2$- norms are proved and confirmed with  numerical experiments.
\end{abstract}

\textbf{Keywords.} interface problem, nonconforming, immersed finite element, unfitted mesh

\textbf{AMS subject classifications.} 65N15, 65N30, 35R05

\section{Introduction}
Let $\Omega\subset\mathbb{R}^2$  be a convex polygonal domain and $\Gamma$  be a $C^2$-smooth interface immersed in $\Omega$.  Without loss of generality, we assume that  $\Gamma$  divides $\Omega$ into two disjoint sub-domains $\Omega^+$ and $\Omega^-$ such that $\Gamma=\partial\Omega^-$, see Figure~\ref{interfacepb} for an illustration. 
We consider  the following second-order elliptic interface problem
\begin{align}
-\nabla\cdot(\beta(x)\nabla u(x))&=f(x)  \qquad\mbox{in } \Omega\backslash\Gamma,\label{p1.1}\\
[u]_{\Gamma}(x)&=0~~~~ \qquad\mbox{on } \Gamma,\label{p1.2}\\
[\beta\nabla u\cdot \textbf{n}]_{\Gamma}(x)&=0~~~~ \qquad\mbox{on } \Gamma,\label{p1.3}\\
u(x)&=0 ~~~~\qquad\mbox{on } \partial\Omega,\label{p1.4}
\end{align}
where  $\textbf{n}(x)$ is the unit normal vector of the interface $\Gamma$ at point $x\in\Gamma$ pointing toward $\Omega^+$, and the notation $[v]_\Gamma$  is defined as
\begin{equation}
[v]_{\Gamma}:=v^+|_{\Gamma}-v^-|_{\Gamma} ~~\mbox{ with } ~~v^\pm=v|_{\Omega^\pm}
\end{equation}
for a  piecewise smooth function $v$. 
The coefficient $\beta(x)$  can be  discontinuous across the interface $\Gamma$ and is assumed to be piecewise smooth
\begin{align}\label{p1.5}
\beta(x)=\beta^+(x) \mbox{ if } x\in\Omega^+ \mbox{ and } \beta(x)=\beta^-(x) \mbox{ if } x\in\Omega^-,
\end{align}
with $\beta^s(x)\in C^1(\overline{\Omega^s})$, $s=+,-$. We also assume that there exist two positive constants $\beta_{min}$ and $\beta_{max}$  such that $\beta_{min}\leq \beta^s(x)\leq\beta_{max}$ for all $x\in \overline{\Omega^s}$, $s=+,-$.

\begin{figure} [htbp]
\centering
\includegraphics[height=0.3\textwidth]{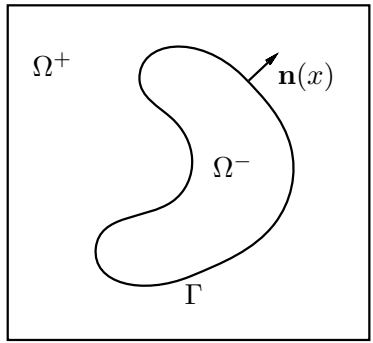}
 \caption{A diagram of the geometry of an interface problem.}\label{interfacepb} 
\end{figure}

It is well-known that traditional isoparametric finite element methods using an interface-fitted mesh can solve the interface problem with optimal convergence rates, see for example \cite{Li2010Optimal,bramble1996finite,chen1998finite,0xu}. For  complicated interfaces or moving interfaces, unfitted meshes, which are not necessarily aligned with interfaces, have some advantages over  interface-fitted meshes. However, traditional  finite element methods using unfitted meshes only achieve suboptimal convergence rates ($O(h^{1/2})$ in the $H^1$ norm and $O(h)$ in the $L^2$ norm) no matter how high degree of the polynomial is used, see  \cite{babuvska1970finite,2014A}.

The design and analysis of finite element methods on unfitted meshes with optimal convergence rates was started in \cite{Barrett1984,Barrett1987}. Since then, many unfitted mesh finite element methods have been developed (see \cite{hansbo2002unfitted,2012An,2018Robust,multi_CHU2010,GuzmanJSC2017,Li2003new} for a few examples).
 Among these methods,  immersed finite element (IFE) methods \cite{li1998immersed,Li2003new,ygong-li,li2004immersed,he2012convergence,He2008,Chang2011An,Guo2018CMA,2019Non_Lin}  are designed to recover the optimal convergence rates of traditional  finite element methods on unfitted meshes while keeping the degrees of freedom and the structure unchanged. 
The basic idea of IFEs is to modify traditional shape functions on interface elements to satisfy  interface conditions approximately.
However, these modifications are done on each interface element independently, which may cause discontinuities of IFE basis functions across interface edges. 
Even for the $P_1$ conforming IFE method, these discontinuities are not negligible \cite{he2012convergence,taolin2015siam} and the optimal convergence rates cannot be achieved if the discontinuities are not treated appropriately. To overcome the difficulty, Lin et al. \cite{taolin2015siam} proposed a partially penalized IFE method where extra penalty terms at interface edges were added to penalize the discontinuity.  For nonconforming IFE (i.e., a modification to the traditional Crouzeix-Raviart element \cite{crouzeix1973conforming} or  the rotated-$Q_1$ element \cite{rannacher1992simple}), we can choose midpoint values or integral-values on edges as degrees of freedom. If we choose the midpoint values of edges as degrees of freedom, the discontinuities of IFE basis functions are also not negligible and the optimal convergence rates can be obtained by adding penalties (see \cite{zhangphd}). In contrast to the case of  midpoint values  as degrees of freedom,  if the integral-values on edges are used  as the degrees of freedom, then the IFE basis functions have less severe discontinuity across interface edges  since the basis functions  maintain the integral-value continuous \cite{zhangphd, Guo2018CMA}. 
It seems  that this choice of degrees of freedom might overcome the difficulty caused by the discontinuities without using penalties. Extensive numerical examples in the literature \cite{Chang2011An, 2019Non_Lin, zhangphd, Guo2018CMA} support this opinion. However, the rigorous proof is missing and the current research tends to improve the analysis of the related algorithms,
as  quoted  in \cite[pp. 96]{zhangphd}: ``How to theoretically prove that the Galerkin IFE scheme with nonconforming rotated $Q_1$ IFE functions using integral-value degrees of freedom does converge optimally is an interesting future research topic.''

In this paper, we show that those nonconforming IFE methods using the integral-value degrees of freedom are not guaranteed to  converge optimally without penalties unless the tangential derivative of the true solution (i.e., $\nabla u\cdot \mathbf{t}$) or the jump of the coefficient $\beta$ is  zero  on the interface (see Theorem~\ref{the_H1_err} and Remark~\ref{remark_suboptimal}).  Furthermore, to validate our theoretical analysis, a nontrivial counter  example with $\nabla u\cdot \mathbf{t}\not=0$ (see {\em Example}~\ref{ex1}) is constructed to show that nonconforming IFE methods using the integral-value degrees without adding penalties may only achieve suboptimal convergence rates (i.e., $O(h^{1/2})$ in the $H^1$ norm and $O(h)$ in the $L^2$ norm).  Note that it is relatively easy to construct an exact solution to satisfy the homogeneous interface conditions (\ref{p1.2})-(\ref{p1.3}) when the exact solution is a constant along the interface (i.e., $\nabla u\cdot \mathbf{t}=0$). To the best of our knowledge, almost all existing numerical examples in the literature satisfy the condition $\nabla u\cdot \mathbf{t}=0$ on the interface and thus the optimal convergence rates are observed (see \cite{Chang2011An, 2019Non_Lin, zhangphd, Guo2018CMA} for example) which is in agreement with our theoretical analysis (see Theorem~\ref{the_H1_err} in Section~\ref{sec_nonconformingIFE}) in this paper.

To achieve the optimal convergence rates, the natural way is to add penalties on interface edges \cite{zhangphd}.   However, the  symmetric  partially penalized IFE methods proposed in \cite{zhangphd} need a manually chosen parameter which is assumed to be large enough. In \cite{refId0},  Guo et al. analyzed a partially penalized IFE methods for elasticity interface problems and derived a lower bound for the penalty parameter. In this paper, we  develop a parameter-free nonconforming IFE method 
by using  a lifting operator defined locally on interface edges. We consider 
both the Crouzeix-Raviart  and  the rotated-$Q_1$  element for solving interface problems with variable coefficients.  
The method is symmetric and the coercivity is ensured without requiring a sufficiently large parameter.
To avoid integrating on curved regions, we also approximate  the interface by line segments connecting the intersection points of the mesh and the interface.  The optimal error estimates are derived rigorously and  are verified  by numerical experiments.  
  
There are other contributions of this paper. First, we  prove that Crouzeix-Raviart IFE basis functions are unisolvent on arbitrary triangles if  the integral-values on edges are used  as degrees of freedom. 
Our recent study in \cite{2021ji_IFE} shows that, for the IFEs using nodal values as degrees of freedom, the maximum angle condition, 
$\alpha_{max}\leq \pi/2$ on interface triangles, is necessary to ensure the  unisolvence of the basis functions. 
The unisolvence of basis functions on arbitrary triangles is a significant advantage of the nonconforming IFEs using integral-value degrees of freedom over the IFEs using nodal values as degrees of freedom. 
Another contribution is a unified proof of the optimal  interpolation error estimates for both the Crouzeix-Raviart and the rotated-$Q_1$  nonconforming IFE spaces for interface problems with piecewise smooth coefficients. Different from multipoint Taylor expansions \cite{Guo2018CMA,zhangphd}, our proof is based on auxiliary functions constructed on interface elements and some useful inequalities developed by Li et al. in \cite{Li2010Optimal} and by Bramble and King in \cite{James1994A} for estimating errors in the region near the interface.  
The other contribution is a new theoretical result that the interpolation polynomial on one side of the interface  can approximate the extensions of the exact solution optimally on the whole element $T$ no matter how small $T\cap\Omega^+$ or $T\cap\Omega^-$    might be (see Theorem~\ref{chazhi_error} in Section~\ref{sec_capabilities}). The result is useful for proving the optimal interpolation error estimates on interface edges (see (\ref{pro_int2b}) in the proof of Lemma~\ref{ener_app}).

The rest of the paper is organized as follows. In Section~\ref{sec_notation}, we introduce some notations,  assumptions,  and basic lemmas that are needed for the analysis. In Section~\ref{sec_IFEspace}, we describe two nonconforming IFE spaces based on the Crouzeix-Raviart  and the rotated-$Q_1$  elements with  integral-value degrees of freedom and prove the unisolvence of IFE basis functions on arbitrary triangles. Furthermore, optimal approximation capabilities of these nonconforming IFE spaces are derived via a new approach. In Section~\ref{sec_nonconformingIFE}, we analyze  the nonconforming IFE method using the integral-value degrees of freedom without adding penalties. In Section~\ref{sec_method2}, we develop a new nonconforming IFE method and derive optimal error estimates in $H^1$- and $L^2$- norms.  Numerical examples are presented  in Section~\ref{sec_num} to validate our theoretical findings.   We conclude in the last section.

\section{Preliminaries}\label{sec_notation}
Throughout the paper we adopt the standard notation $W^k_p(\Lambda)$ for Sobolev spaces on a domain $\Lambda$ with the norm $\|\cdot\|_{W^k_p(\Lambda)}$ and the seminorm $|\cdot|_{W^k_p(\Lambda)}$. Specially, $W^k_2(\Lambda)$ is denoted by $H^{k}(\Lambda)$ with the norm  $\|\cdot\|_{H^{k}(\Lambda)}$ and the seminorm $|\cdot|_{H^{k}(\Lambda)}$. As usual $H_0^1(\Lambda)=\{v\in H^1(\Lambda) : v=0 \mbox{ on }\partial \Lambda\}$. For a domain $\Lambda$, we define 
$\Lambda^s:=\Lambda\cap \Omega^s$, $s=+,-$ and a space
\begin{equation}\label{def_H2}
\widetilde{H}^2(\Lambda):=\{v\in L^2(\Lambda) : v|_{\Lambda^s} \in H^2(\Lambda^s), s=+,-,~ [v]_{\Gamma\cap\Lambda}=0,~ [\beta\nabla v\cdot \textbf{n}]_{\Gamma\cap\Lambda}=0\}
\end{equation}
when $\Lambda^s\not =\emptyset$, $s=+,-$. Obviously, $\widetilde{H}^2(\Lambda)\subset H^1(\Lambda)$.
The space $\widetilde{H}^2(\Lambda)$ is equipped with the norm $\|\cdot\|_{H^2(\Lambda^+\cup\Lambda^-)}$ and the  semi-norm $|\cdot|_{H^2(\Lambda^+\cup\Lambda^-)}$ satisfying
$$
\|\cdot\|^2_{H^2(\Lambda^+\cup\Lambda^-)}=\|\cdot\|^2_{H^2(\Lambda^+)}+\|\cdot\|^2_{H^2(\Lambda^-)}, \quad|\cdot|^2_{H^2(\Lambda^+\cup\Lambda^-)}=|\cdot|^2_{H^2(\Lambda^+)}+|\cdot|^2_{H^2(\Lambda^-)}.
$$
By integration by parts, we can immediately derive the following weak formulation of  (\ref{p1.1})-(\ref{p1.4}): find $u\in H^1_0(\Omega)$ such that
\begin{equation}\label{weakform}
a(u,v):=\int_{\Omega}\beta(x)\nabla u\cdot\nabla vdx=\int_\Omega fvdx\qquad \forall v\in H_0^1(\Omega).
\end{equation}
We have the following regularity theorem for the weak solution (see \cite{2012Uniform} for the case of piecewise smooth  coefficients and \cite{huang2002some,multi_CHU2010} for the case of piecewise constant coefficients).
\begin{theorem}\label{regular_assumption}
If $f\in L^2(\Omega)$, then (\ref{p1.1})-(\ref{p1.4}) has a unique solution $u\in \widetilde{H}^2(\Omega)$ satisfying the following a priori estimate
\begin{equation}\label{regular}
\|u\|_{H^2(\Omega^+\cup\Omega^-)}\leq C\|f\|_{L^2(\Omega)},
\end{equation}
where $C$ is a positive constant depending only on $\Omega$, $\Gamma$ and $\beta$.
\end{theorem}

 Let $ \{\mathcal{T}_h\}_{h>0}$ be a family of triangular or rectangular subdivisions of $\Omega$ such that no vertex of any element lies in the interior of an edge of another element. The diameter of $T\in\mathcal{T}_h$  is denoted by $h_T$. We define  $h=\max_{T\in\mathcal{T}_h}h_T$ and assume that $\mathcal{T}_h$ is shape regular, i.e., for every $T$, there exists a positive constant $\varrho$ such that  $ h_T\leq \varrho r_T$ where $r_T$ is the diameter of the largest circle inscribed in $T$.  
%
%
Denote $\mathcal{E}_h$  as the set of edges of the subdivision, and let $\mathcal{E}^\circ_h$ and $\mathcal{E}^b_h$  be the sets of interior edges and boundary edges. We adopt the convention that elements $T\in\mathcal{T}_h$ and edges $e\in\mathcal{E}_h$ are open sets. Then, the sets of interface elements and interface edges are defined as
\begin{equation*}
\mathcal{T}_h^\Gamma :=\{T\in\mathcal{T}_h :  T\cap \Gamma\not = \emptyset\} \quad\mbox{ and }\quad \mathcal{E}_h^\Gamma:=\{e\in \mathcal{E}_h : e \cap \Gamma\not = \emptyset\}, 
\end{equation*}
and the sets of non-interface elements and non-interface edges are $\mathcal{T}^{non}_h=\mathcal{T}_h\backslash\mathcal{T}_h^{\Gamma}$ and $\mathcal{E}^{non}_h=\mathcal{E}_h\backslash\mathcal{E}_h^{\Gamma}$.
We can always refine the mesh near the interface to  satisfy the following assumption.
\begin{assumption}\label{assum_2}
The interface $\Gamma$ does not intersect the boundary of any interface element at more than two points. The interface $\Gamma$ does not intersect  the closure $\overline{e}$ for any $e\in\mathcal{E}_h$  at more than one point.
\end{assumption}

The interface $\Gamma$ is approximated by $\Gamma_h$ that is composed of all the line segments connecting the intersection points of the boundaries of interface elements and the interface. In addition, we assume that $\Gamma_h$ divides $\Omega$ into two disjoint sub-domains $\Omega^+_h$ and $\Omega^-_h$ such that $\Gamma_h=\partial \Omega^-_h$.  

Given an  interface element $T\in\mathcal{T}_h^\Gamma$, we denote the intersection points of $\Gamma$ and $\partial T$ by $D$ and $E$. The straight line $DE$ divides $T$ into $T^+_h=T\cap \Omega^+_h$ and $T^-_h=T\cap \Omega^-_h$, see Figure~\ref{interface_ele} for an illustration.

Let $\textbf{n}_h(x)$  be the unit normal vector of $\Gamma_h$ pointing toward $\Omega^+_h$.  The  tangent vector of $\Gamma_h$ can be defined as $\textbf{t}_h(x)=R_{\pi/2}\textbf{n}_h(x)$, where
$R_\alpha$ is  a  rotation matrix 
\begin{equation*}
R_\alpha=
\left[
\begin{aligned}
&\cos\alpha &-\sin\alpha\\
&\sin\alpha  &\cos\alpha
\end{aligned}
\right].
\end{equation*}

Denote $\mbox{dist}(x,\Gamma)$ as  the distance between a point $x$ and the interface $\Gamma$, and  $U(\Gamma,\delta)=\{x\in\mathbb{R}^2: \mbox{dist}(x,\Gamma)< \delta\}$  as the neighborhood of $\Gamma$ of thickness $\delta$. 
Define the meshsize of $\mathcal{T}_h^\Gamma$ by
\begin{equation}\label{hgamma}
h_\Gamma:=\max_{T\in\mathcal{T}_h^\Gamma}h_T.
\end{equation}
It is obvious that $h_\Gamma\leq h$ and $\bigcup_{T\in\mathcal{T}_h^\Gamma} T\subset U(\Gamma,h_\Gamma)$.

We also define a signed distance function $\rho(x)$ with $\rho(x)|_{\Omega^+}=\mbox{dist}(x,\Gamma)$ and $\rho(x)|_{\Omega^-}=-\mbox{dist}(x,\Gamma)$. There exists a constant $\delta_0>0$ such that $\rho(x)$ is well-defined in $U(\Gamma,\delta_0)$ and $\rho(x)\in C^2(U(\Gamma,\delta_0))$ (see \cite{foote1984regularity}).

\begin{assumption}\label{assumption_delta}
We assume that $h_\Gamma<\delta_0$ so that  $\overline{T}\subset U(\Gamma,\delta_0)$ for all $T\in\mathcal{T}_h^\Gamma$.
\end{assumption}

We extend the coefficients $\beta^s(x)$, $s=+,-$ smoothly to slightly larger domains $\Omega_e^s:=\Omega^s\cup U(\Gamma,\delta_0)$, $s=+,-$ such that 
\begin{equation}\label{ext_beta}
\beta^s(x)\in C^1(\overline{\Omega_e^s})~~\mbox{ and }~~ \beta^e_{min}\leq \beta^s(x)\leq\beta^e_{max},~~s=+,-,
\end{equation}
where the constants $\beta^e_{min}$ and $\beta^e_{max}$ depend on $\Gamma$, $\beta^\pm$ and $\delta_0$.
Thus, there exists a constant $C_\beta$ such that
\begin{equation}\label{new_vari_deri}
\|\nabla \beta^s\|_{L^\infty(\overline{\Omega_e^s})}\leq C_\beta,\qquad s=+,-.
\end{equation}
By using the signed distance function $\rho$, we can evaluate the unit normal and tangent vectors of the interface as 
\begin{equation}\label{n_and_t}
   \textbf{n}(x)=\nabla \rho,~~~~ \textbf{t}(x) =\left(-\frac{\partial\rho}{\partial x_2},\frac{\partial\rho}{\partial x_1}\right)^T,
\end{equation}
which are well-defined in the region $U(\Gamma,\delta_0)$.
We note that the functions $\textbf{n}_h(x)$ and $\textbf{t}_h(x)$ are also viewed as piecewise constant vectors defined on interface elements.
On each interface element $T\in\mathcal{T}_h^\Gamma$, since $\Gamma$ is in  $C^2$, by  Rolle's Theorem, there exists at least one point $x^*\in\Gamma\cap T$, see Figure~\ref{interface_ele},  such that
\begin{equation}\label{rolle}
\textbf{n}(x^*)=\textbf{n}_h(x^*) \quad \mbox{ and }\quad \textbf{t}(x^*)=\textbf{t}_h(x^*).
\end{equation}
Since $\rho(x)\in C^2(U(\Gamma,\delta_0))$, we have
\begin{equation}\label{nt_smooth}
\textbf{n}(x)\in \left(C^1(\overline{T})\right)^2\quad \mbox{ and }\quad\textbf{t}(x)\in \left(C^1(\overline{T})\right)^2\quad \forall T\in\mathcal{T}_h^\Gamma.
\end{equation}
Using  Taylor's expansion  at $x^*$, we further have
\begin{equation}\label{error_nt}
\|\textbf{n}-\textbf{n}_h\|_{L^\infty(T)}\leq Ch_T\quad \mbox{ and }\quad\|\textbf{t}-\textbf{t}_h\|_{L^\infty(T)}\leq Ch_T\quad \forall T\in\mathcal{T}_h^\Gamma.
\end{equation}
The following lemma presents a $\delta$-strip argument that will be used for the error estimate in the region near the interface (see Lemma 2.1 in \cite{Li2010Optimal}).
\begin{lemma}\label{strip}
Let $\delta$ be sufficiently small.  Then it holds  for any $v\in H^1(\Omega)$ that 
\begin{equation*}
\|v\|_{L^2(U(\Gamma,\delta))}\leq C\sqrt{\delta}\,  \| v\|_{H^1(\Omega)}.
\end{equation*}
Furthermore,  if $v|_{\Gamma}=0$, then there holds
\begin{equation*}
\|v\|_{L^2(U(\Gamma,\delta))}\leq C\delta \, \|\nabla v\|_{L^2(U(\Gamma,\delta))}.
\end{equation*}
\end{lemma}
Recalling $T^s=T\cap \Omega^s,~ T_h^s=T\cap\Omega_h^s$, $s=+,-$ for all $T\in \mathcal{T}_h^\Gamma$, we define
\begin{equation}\label{tri}
T^\triangle:=(T^-\cap T_h^+)\cup(T^+\cap  T_h^-).
\end{equation}
Since $\Gamma$ is in $C^2$, we have $|T^\triangle|\leq Ch_T^3$.
We shall need the following estimate on the region $T^\triangle$ (see Lemma 2 in \cite{James1994A}).
\begin{lemma}\label{lem_h3}
Assume  that $v\in H^1(T)$ and $T\in\mathcal{T}_h^\Gamma$. Then there is a constant $C$, independent of $h$  and the interface location relative to the mesh, such that
\begin{equation*}
\|v\|_{L^2(T^\triangle)}^2\leq C(h_T^2\|v\|^2_{L^2(\Gamma\cap T)}+h_T^4\|\nabla v\|^2_{L^2(T^\triangle)}).
\end{equation*}
\end{lemma}

\section{Nonconforming  IFE spaces and their properties}\label{sec_IFEspace}
In this section, we describe nonconforming IFE spaces based on the Crouzeix-Raviart  element  or  the rotated-$Q_1$ element and present their properties. To begin with, we define IFE shape function spaces.
On a non-interface element $T\in\mathcal{T}_h^{non}$,  we use the traditional  shape function space
\begin{equation*}
V_h(T)=\left\{
\begin{aligned}
&\mbox{Span} \{1,x_1,x_2\},\qquad &&\mbox{ for the Crouzeix-Raviart element ($T$ is a triangle), }\\
&\mbox{Span} \{1,x_1,x_2, x_1^2-(\kappa_T x_2)^2\}, &&\mbox{ for the rotated-$Q_1$ element ($T$ is a rectangle),}
\end{aligned}\right.
\end{equation*}
where $\kappa_T=|e_1|/|e_2|$, $e_1$ and $e_2$ are edges of the rectangle and parallel to the $x_1$-axis and the $x_2$-axis, respectively.
On an interface element $T\in\mathcal{T}_h^{\Gamma}$, the  IFE shape function space $S_h(T)$ is defined as the set of the following functions
\begin{equation}\label{shape1}
\phi(x)=\left\{
\begin{aligned}
\phi^+(x)\in V_h(T)\quad \mbox{ if } x=(x_1,x_2)^T\in T_h^+,\\
\phi^-(x)\in V_h(T)\quad \mbox{ if } x=(x_1,x_2)^T\in T_h^-,
\end{aligned}
\right.
\end{equation}
 satisfying
\begin{align}
&[\phi]_{\Gamma_h\cap T}(x)=\phi^+(x)-\phi^-(x)=0\quad \forall x \in \Gamma_h\cap T, \label{shape2.1}\\
&\beta_c^+ (\nabla \phi^+ \cdot \textbf{n}_h)(x_p)-\beta_c^- (\nabla \phi^- \cdot \textbf{n}_h)(x_p)=0, \label{shape2}
\end{align}
where $x_p$ is an arbitrary point on  $\Gamma_h\cap T$ and the constants $\beta^+_c$ and $\beta^-_c$ are chosen such that 
\begin{equation}\label{shape_condi}
\|\beta^s(x)-\beta_c^s\|_{L^\infty(T)}\leq Ch_T,\quad s=+,-.
\end{equation}
Actually, we can choose $\beta^s_c=\beta^s(x_c^s)$ with arbitrary $x_c^s\in T$, $s=+,-$, to satisfy the condition (\ref{shape_condi}) since we know that $\beta^s(x)\in C^1(\overline{T})$, $s=+,-$ from (\ref{ext_beta}).
\begin{remark}\label{remark_cc}
For the Crouzeix-Raviart element, the condition (\ref{shape2.1}) is equivalent to 
$$\phi^+(D)=\phi^-(D),\quad \phi^+(E)=\phi^-(E),$$
since $\phi^s(x)$, $s=+,-$,  are linear functions. 
For the rotated-$Q_1$ element,  we can write $\phi^s(x)$ as 
$$\phi^s(x)=a^s+b^s x_1+c^s x_2+d^s(x_1^2-(\kappa_T x_2)^2), \quad x=(x_1,x_2)^T,~~s=+,-,$$
where $a^s, b^s, c^s, d^s$, $s=+,-$, are constants.  If we define a functional $d: V_h(T)\rightarrow \mathbb{R}$ as
\begin{equation}\label{def_d}
d(\phi^s)= \frac{1}{2}\frac{\partial^2\phi^s}{\partial x_1^2}=\frac{1}{\sqrt{(4\kappa_T^4+4)|T|}}|\phi^s|_{H^2(T)},
\end{equation}
then $d^s=d(\phi^s)$. Similar to Lemma~2.1 in \cite{He2008},  the condition (\ref{shape2.1}) is equivalent to 
$$\phi^+(D)=\phi^-(D),\quad \phi^+(E)=\phi^-(E),\quad d(\phi^+)=d(\phi^-).$$ 
\end{remark}

\begin{remark}
For the Crouzeix-Raviart element, $\beta_c^s \nabla \phi^s \cdot \textbf{n}_h$, $s=+,-$, are constants on the interface element. Thus, the condition (\ref{shape2}) is equivalent to
\begin{equation}\label{cr_n}
\beta_c^+ (\nabla \phi^+ \cdot \textbf{n}_h)(x)=\beta_c^- (\nabla \phi^- \cdot \textbf{n}_h)(x)\qquad \forall x\in\Gamma_h\cap T.
\end{equation}
 However, for the rotated-$Q_1$ element, the relation (\ref{cr_n}) is no longer valid.  In \cite{2019Non_Lin}, the authors weakly enforce the continuity by using the following condition
 \begin{equation*}
 \int_{\Gamma_h\cap T}\beta_c^+ (\nabla \phi^+ \cdot \textbf{n}_h)-\beta_c^- (\nabla \phi^- \cdot \textbf{n}_h)ds=0
 \end{equation*}
 which is a particular case of  (\ref{shape2}) since there exists a point $x_p\in \Gamma_h\cap T$ such that 
  \begin{equation*}
 \int_{\Gamma_h\cap T}(\beta_c^+ \nabla \phi^+-\beta_c^- \nabla \phi^-) \cdot \textbf{n}_hds=|\Gamma_h\cap T|(\beta_c^+ \nabla \phi^+-\beta_c^- \nabla \phi^-)(x_p) \cdot \textbf{n}_h.
 \end{equation*}
\end{remark}

Let $\mathcal{I}=\{1,2,3\}$ for the Crouzeix-Raviart element and $\mathcal{I}=\{1,2,3,4\}$ for the rotated-$Q_1$ element. The degrees of freedom are defined as the mean values over edges
\begin{equation*}
N_i(\phi):=\frac{1}{|e_i|}\int_{e_i}\phi ds,\qquad  i\in\mathcal{I},
\end{equation*}
where $e_i$, $i\in\mathcal{I}$ are edges of the element $T$, and $|e_i|$ denotes the length of the edge $e_i$. 
On an interface element $T\in\mathcal{T}_h^\Gamma$,  the  immersed finite element is defined as $(T, S_h(T), \Sigma_T)$ with $\Sigma_T=\{ N_i , i\in\mathcal{I}\}$.

\textbf{The nonconforming IFE spaces}. The nonconforming IFE space $V_h^{\rm IFE}$ is defined as the set of all functions satisfying
\begin{equation*}
\left\{
\begin{aligned}
&\phi|_T \in S_h(T) ~~~~~\forall  T\in\mathcal{T}_h^\Gamma,\\
&\phi|_T \in V_h(T) ~~~~~\forall  T\in\mathcal{T}_h^{non},\\
&\int_e[\phi]_e ds=0~~~~~\forall e\in\mathcal{E}^\circ_h.
\end{aligned}
\right.
\end{equation*}
We also need a space for homogeneous boundary conditions
$$V_{h,0}^{\rm IFE}:=\left \{v\in V_h^{\rm IFE} : \int_evds=0~~ \forall e\in \mathcal{E}_h^b \right \}.$$

\subsection{The unisolvence of IFE basis functions}

It was proved in \cite{Guo2018CMA} that  the function $\phi\in S_h(T)$ is uniquely determined by $N_i(\phi)$, $i=1,2,3,4$ for the rotated-$Q_1$ element, and $i=1,2,3$ for the Crouzeix-Raviart element when the interface element is an  isosceles right triangle. Now we prove that the result is also valid for arbitrary triangles in the following lemma.  Note that for the IFEs using nodal values as degrees of freedom, the maximum angle condition, $\alpha_{max}\leq \pi/2$  on interface triangles, is necessary to ensure the  unisolvence (see \cite{2021ji_IFE}). This property of  the unisolvence of basis functions is one of advantages of nonconforming IFEs compared with the IFEs using nodal values as degrees of freedom.

\begin{figure} [htbp]
\centering
\includegraphics[height=0.25\textwidth]{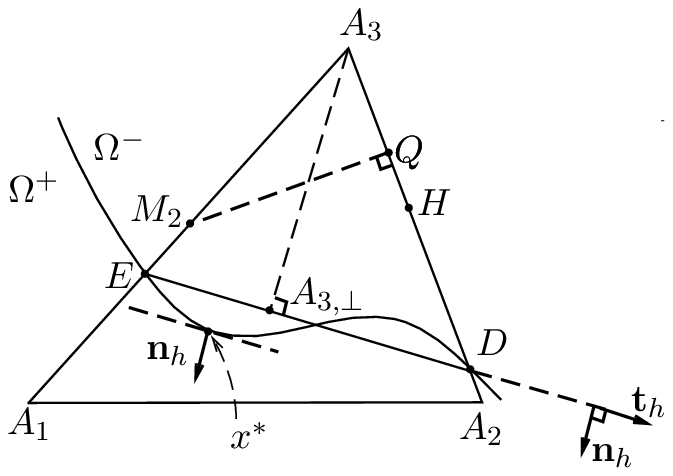}
\quad
\includegraphics[height=0.25\textwidth]{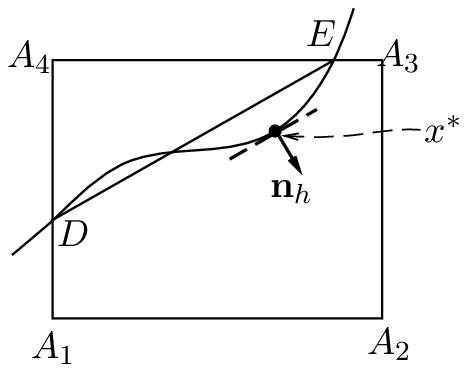}
 \caption{Interface elements.}\label{interface_ele} 
\end{figure}

\begin{lemma}\label{lem_uniq}
Let $T$ be an arbitrary interface triangle. For the Crouzeix-Raviart element, the function $\phi\in S_h(T)$ is uniquely determined by $N_i(\phi)$, $i=1,2,3$.
\end{lemma}
\begin{proof}
We follow the argument proposed in \cite{Guo2018CMA,GuoIMA2019}. Consider a  triangle $\triangle A_1A_2A_3$ with edges $e_1=\overline{A_2A_3}$, $e_2=\overline{A_1A_3}$ and $e_3=\overline{A_1A_2}$. The interface $\Gamma$ cuts $e_1$ and $e_2$ at points $D$ and $E$, see Figure~\ref{interface_ele} for an illustration. Without loss of generality, we  assume $T_h^-=\triangle EDA_3$ since the case $T_h^+=\triangle EDA_3$ can be treated by reversing $\beta^+_c$ and $\beta^-_c$.  Let  $\lambda_i(x)$, $i=1,2,3$, be basis functions in $V_h(T)$ such that $\frac{1}{|e_j|}\int_{e_j}\lambda_i(x)dx=\delta_{ij}$ for all $j\in\{1,2,3\}$, where $\delta_{ij}$ is the Kronecker function. Using (\ref{shape2.1}) and $|e_3|^{-1}\int_{e_3}\phi ds=N_3(\phi)$,  we can write the IFE shape function $\phi(x)$ in (\ref{shape1})  as
\begin{equation}\label{pro_phi}
\phi(x)=\left\{
\begin{aligned}
&\phi^+(x)=c_1\lambda_1(x)+c_2\lambda_2(x)+N_3(\phi)\lambda_3(x)\quad &&\mbox{ if } x=(x_1,x_2)^T\in T_h^+,\\
&\phi^-(x)=\phi^+(x)+c_0\textbf{n}_h\cdot \overrightarrow{Dx}\quad &&\mbox{ if } x=(x_1,x_2)^T\in T_h^-,
\end{aligned}
\right.
\end{equation}
where $c_0,c_1,c_2$ are unknowns. Applying the condition (\ref{shape2}),  the unknown $c_0$  can be expressed as 
\begin{equation}\label{exc_0}
c_0=(\beta_c^+/\beta_c^--1)\nabla \phi^+\cdot\textbf{n}_h=(\beta_c^+/\beta_c^--1) (c_1\nabla\lambda_1+c_2\nabla\lambda_2+N_3(\phi)\nabla\lambda_3)\cdot\textbf{n}_h.
\end{equation}
Substituting (\ref{exc_0}) into (\ref{pro_phi}) and using  $N_i(\phi)=|e_i|^{-1}\int_{e_i}\phi ds$, $i=1,2$, we obtain the following linear system of equations for other coefficients
(see \cite{Guo2018CMA,GuoIMA2019} for details),
\begin{equation}\label{unique_shape_eq}
(\mathbf{I}+(\beta^+_c/\beta^-_c-1)\boldsymbol{\delta}\boldsymbol{\gamma}^T)\boldsymbol{c}=\boldsymbol{b},
\end{equation}
where 
\begin{equation}\label{shap_eq}
\begin{aligned}
&\boldsymbol{\delta}=\left(|e_1|^{-1}\int_{\overline{A_3D}}L(x)ds, |e_2|^{-1}\int_{\overline{A_3E}}L(x)ds\right)^T,  ~~L(x)=\textbf{n}_h\cdot\overrightarrow{Dx},\\
&\boldsymbol{c}=(c_1, c_2)^T,~~\boldsymbol{\gamma}=\left(\nabla \lambda_1\cdot\textbf{n}_h, \nabla \lambda_2\cdot\textbf{n}_h\right)^T,\\
&\boldsymbol{b}=\left(N_1(\phi)-\frac{(\beta^+_c/\beta^-_c-1)N_3(\phi)\nabla \lambda_3\cdot\textbf{n}_h}{|e_1|} \int_{\overline{A_3D}}L(x)ds,\right.\\
&~~~~~~~~~~~~\left.N_2(\phi)-\frac{(\beta^+_c/\beta^-_c-1)N_3(\phi)\nabla \lambda_3\cdot\textbf{n}_h}{|e_2|} \int_{\overline{A_3E}}L(x)ds \right)^T.
\end{aligned}
\end{equation}
Set $k_1=|A_3D||e_1|^{-1}$ and $k_2=|A_3E||e_2|^{-1}$. Let $M_i$ be the midpoint of the edge $e_i$, $i=1,2,3$, and $Q$ be  the orthogonal projection of $M_2$ onto the line $A_2A_3$. We can  find out $\boldsymbol{\gamma}(1)$ and $\boldsymbol{\delta}(1)$  as below
\begin{equation*}
\begin{aligned}
\boldsymbol{\gamma}(1)&=\nabla\lambda_1\cdot \textbf{n}_h= |M_2Q|^{-1}\overrightarrow{M_2Q}|M_2Q|^{-1}\cdot \textbf{n}_h=|M_2Q|^{-1}R_{\pi/2}\left(\overrightarrow{M_2Q}|M_2Q|^{-1}\right)\cdot R_{\pi/2}(\textbf{n}_h)\\
&=|M_2Q|^{-1}\overrightarrow{A_2A_3}|A_2A_3|^{-1}\cdot \textbf{t}_h=\left(\frac{1}{2}|e_2|\sin\angle A_3\right)^{-1}|e_1|^{-1}\overrightarrow{A_2A_3}\cdot \textbf{t}_h,
\end{aligned}
\end{equation*}
and
\begin{equation*}
\begin{aligned}
\boldsymbol{\delta}(1)=|e_1|^{-1}\int_{\overline{A_3D}}\textbf{n}_h\cdot\overrightarrow{Dx}ds=|e_1|^{-1}|A_3D|\textbf{n}_h\cdot\overrightarrow{DH}=-\frac{1}{2}k_1|A_3A_{3,\perp}|,
\end{aligned}
\end{equation*}
where $ \angle A_3\in (0,\pi)$, $H$ is the midpoint of the line segment $\overline{A_3D}$, and $A_{3,\perp}$ is the orthogonal projection of $A_3$ onto the line $DE$. 
Thus,
\begin{equation*}
\begin{aligned}
\boldsymbol{\gamma}(1)\boldsymbol{\delta}(1) &=\nabla\lambda_1\cdot \textbf{n}_h|e_1|^{-1}\int_{\overline{A_3D}}\overrightarrow{Dx}ds=-k_1\overrightarrow{A_2A_3}\cdot \textbf{t}_h |A_3A_{3,\perp}|\left(|e_1||e_2|\sin\angle A_3\right)^{-1}\\
&=-\overrightarrow{DA_3}\cdot \textbf{t}_h |A_3A_{3,\perp}|\left(|e_1||e_2|\sin\angle A_3\right)^{-1}.
\end{aligned}
\end{equation*}
Analogously, we have 
\begin{equation*}
\boldsymbol{\gamma}(2)\boldsymbol{\delta}(2)=\nabla\lambda_2\cdot \textbf{n}_h|e_2|^{-1}\int_{\overline{A_3E}}\overrightarrow{Dx}ds=-\overrightarrow{A_3E}\cdot \textbf{t}_h |A_3A_{3,\perp}|\left(|e_1||e_2|\sin\angle A_3\right)^{-1}.
\end{equation*}
Therefore, 
\begin{equation*}
\boldsymbol{\gamma}^T\boldsymbol{\delta}=\overrightarrow{ED}\cdot\textbf{t}_h|A_3A_{3,\perp}|\left(|e_1||e_2|\sin\angle A_3\right)^{-1}=|DE||A_3A_{3,\perp}|\left(|e_1||e_2|\sin\angle A_3\right)^{-1}.
\end{equation*}
As long as  \ $ \angle A_3ED\in (0,\pi)$,  it is true that $|A_3A_{3,\perp}|=k_2|e_2|\sin\angle A_3ED$, which together with  the relations $|DE|\sin^{-1}\angle A_3=k_1|e_1|\sin^{-1}\angle A_3ED$ yields
\begin{equation*}
\boldsymbol{\gamma}^T\boldsymbol{\delta}= k_1|e_1|\left(\sin\angle A_3ED\right)^{-1}k_2|e_2|\left(\sin\angle A_3ED\right)\left(|e_1||e_2|\right)^{-1}=k_1k_2\in [0,1].
\end{equation*}
From the above inequality, we have
\begin{equation}\label{est_max}
1+(\beta^+_c/\beta^-_c-1)\boldsymbol{\gamma}^T\boldsymbol{\delta}\geq \min(1,\beta^+_c/\beta^-_c)\geq\beta_{min}/\beta_{max}>0.
\end{equation}
Hence, by the well-known  Sherman-Morrison formula, the linear system (\ref{unique_shape_eq}) has a unique solution
\begin{equation}\label{soltu}
\boldsymbol{c}=\boldsymbol{b}-\frac{(\beta^+_c/\beta^-_c-1)(\boldsymbol{\gamma}^T\boldsymbol{b})\boldsymbol{\delta}}{1+(\beta^+_c/\beta^-_c-1)\boldsymbol{\gamma}^T\boldsymbol{\delta}},
\end{equation}
 which completes the proof of the lemma. 
\end{proof}

\begin{remark}
In \cite{Guo2018CMA}, the authors also consider the case that the curved interface is not discretized and the interface condition (\ref{p1.3}) is enforced on a point on the exact interface, i.e., 
replacing $T_h^+$ and $T_h^-$ in (\ref{shape1}) by $T^+$ and $T^-$, respectively,  and replacing
(\ref{shape2}) by
$$\beta_c^+ \nabla \phi^+ \cdot \textbf{n}(F)-\beta_c^- \nabla \phi^- \cdot \textbf{n}(F)=0,$$
where $F$ is a point on $\Gamma\cap T$. Thus, the result for $c_0$ in (\ref{exc_0}) involves $\mathbf{n}(F)\cdot\mathbf{n}_h$ (see \cite[Eq. (4.7)]{Guo2018CMA}).  In this paper, we do not consider this case. We rigorously analyze the error caused by discretizing the curved interface by line segments.
\end{remark}

\subsection{Optimal approximation capabilities of IFE spaces}\label{sec_capabilities}
On each element $T\in\mathcal{T}_h$, define a local interpolation operator  $I_{h,T}: W(T)\rightarrow V_h(T)$ such that
\begin{equation*}
N_i(I_{h,T}v)=N_i(v)\quad \forall i\in\mathcal{I},
\end{equation*}
where $W(T)=\{v: N_i(v), i \in\mathcal{I} \mbox{ are well defined}\}$.
Similarly, on each interface element $T\in\mathcal{T}_h^\Gamma$, define  $I^{\rm IFE}_{h,T}: W(T) \rightarrow S_h(T)$ such that
\begin{equation}\label{local_IFE_inter}
N_i(I^{\rm IFE}_{h,T}v)=N_i(v)\quad \forall i\in\mathcal{I}.
\end{equation}
 The global IFE interpolation operator is defined by $I_h^{\rm IFE}: H^1(\Omega)\rightarrow V_h^{\rm IFE}$  such that 
\begin{equation*}
(I_h^{\rm IFE}v)|_{T}=\left\{
\begin{aligned}
&I_{h,T}^{\rm IFE}v\quad&&\mbox{ if } ~T\in\mathcal{T}_h^\Gamma,\\
&I_{h,T} v&&\mbox{ if }~ T\in\mathcal{T}_h^{non}.\\
\end{aligned}\right.
\end{equation*}
%
 For simplicity, define $v^s:=v|_{\Omega^s}$, $s=+,-$ for all $v\in L^2(\Omega)$. With a small ambiguity of notation, given a function $v_h\in S_h(T)$, we define $v_h^s\in V_h(T)$, $s=+,-$ such that
\begin{equation}\label{def_pm}
v_h^s=v_h|_{T_h^s},\quad s=+,-.
\end{equation}
To show that functions in $V_h^{\rm IFE}$ can approximate a function $v$ in $\widetilde{H}^2(\Omega)$ optimally,  we need to  interpolate extensions of $v^s$, $s=+,-$. It is well-known that  (see \cite{Gilbargbook}) for any $v\in \widetilde{H}^2(\Omega)$ there exist extensions $v_E^s\in H^2(\Omega)$, $s=+,-$ such that
\begin{equation}\label{extension}
 v_E^s|_{\Omega^s}=v^s~~\mbox{ and }~~\| v_E^s\|_{H^2(\Omega)}\leq C\|v^s\|_{H^2(\Omega^s)},\quad s=+,-.
\end{equation}

The next two theorems state the optimal approximation properties of the immersed finite element spaces. The proofs are technical and thus are presented in  the appendix. 
\begin{theorem}\label{chazhi_error}
For any $v\in  \widetilde{H}^2(\Omega)$, there exists a constant $C$ independent of $h$ and the interface location relative to the mesh such that
\begin{equation}\label{fuzhu}
\sum_{T\in\mathcal{T}_h^\Gamma}h_T^{2(m-1)}|v_E^s-(I_h^{\rm IFE}v)^s|^2_{H^m(T)}\leq Ch_\Gamma^{2}\|v\|^2_{H^2(\Omega^+\cup\Omega^-)},~~ m=0,1,2,~~s=+,-.
\end{equation}
\end{theorem}
\begin{proof}
See Appendix~\ref{app_theo1}.
\end{proof}
The above result shows that the interpolation polynomial on one side of the interface  can approximate the extensions of the exact solution optimally on the whole element $T$ no matter how small $T\cap\Omega^+$ or $T\cap\Omega^-$  might be. This is the key in  deriving the optimal error estimates on interface edges; see (\ref{pro_int2b}) in the proof of Lemma~\ref{ener_app}.

Taking into account the mismatch of $\Gamma$ and $\Gamma_h$, we can prove the optimal approximation capabilities of the nonconforming IFE spaces.
\begin{theorem}\label{theorem_interpolation}
For any $v\in  \widetilde{H}^2(\Omega)$, there exists a constant $C$  independent of $h$ and the interface location relative to the mesh such that
\begin{equation}\label{pro_int1a}
\sum_{T\in\mathcal{T}_h}|v-I_h^{\rm IFE}v|^2_{H^m(T)}\leq Ch^{4-2m}\|v\|^2_{H^2(\Omega^+\cup\Omega^-)},\quad m=0,\,1.
\end{equation}
\end{theorem}
\begin{proof}
See Appendix~\ref{app_theo2}.
\end{proof}

\section{Analysis of the nonconforming IFE method without penalties}\label{sec_nonconformingIFE}

In this section, we analyze the nonconforming IFE method without penalties  which is obtained from (\ref{weakform}) by simply replacing the Sobolev space $H_0^1(\Omega)$  with the nonconforming IFE space $V_{h,0}^{\rm IFE}$ (see \cite{Chang2011An,2019Non_Lin}).
%
%
 The method reads: find $u_h\in V_{h,0}^{\rm IFE}$ such that
\begin{equation}\label{method1}
a_h(u_h,v_h):=\sum_{T\in\mathcal{T}_h}\int_T\beta(x)\nabla u_h\cdot\nabla v_hdx=\int_\Omega fv_hdx\qquad \forall v_h\in V_{h,0}^{\rm IFE}.
\end{equation}
We will show that the nonconforming IFE method without penalties is not guaranteed to converge optimally unless $[\beta]_\Gamma\nabla u\cdot\textbf{t}=0$ on $\Gamma$.

%
It is easy to see that $a_h(\cdot,\cdot)=a(\cdot,\cdot)$ on $H^1_0(\Omega)$ and it is positive-definite on $V_{h,0}^{\rm IFE}$ because $a_h(v_h,v_h)=0$, $v_h\in V_{h,0}^{\rm IFE}$ implies $v_h=0$.  Thus, the discrete problem (\ref{method1}) has a unique solution. We define the energy norm 
$$\|v\|_{a_h}:=\sqrt{a_h(v,v)}\qquad\forall v\in V_{h,0}^{\rm IFE}+H_0^1(\Omega)$$ 
and quote the following well-known second Strang lemma (see Lemma 10.1.9 in \cite{brenner2008mathematical}).
\begin{lemma}\label{lem_strang}
Let $u$ and $u_h$ be the solutions of  (\ref{weakform}) and (\ref{method1}), respectively. Then  
\begin{equation}\label{strang}
\|u-u_h\|_{a_h}\leq C\left\{\inf_{v_h\in V_{0,h}^{\rm IFE}}\|u-v_h\|_{a_h}+\sup_{w_h\in V_{h,0}^{\rm IFE}\backslash\{0\}}\frac{|a_h(u-u_h,w_h)|}{\|w_h\|_{a_h}}\right\}.
\end{equation}
\end{lemma}
Since $u\in \widetilde{H}^2(\Omega)$,  Theorem~\ref{theorem_interpolation} implies 
\begin{equation}\label{H1_est}
\inf_{v_h\in V_{h,0}^{\rm IFE}}\|u-v_h\|_{a_h}\leq \|u-I_h^{\rm IFE}u\|_{a_h}\leq Ch\|u\|_{H^2(\Omega^+\cup\Omega^-)}.
\end{equation}
For the second term on the right-hand side of (\ref{strang}), we have
\begin{equation*}
\begin{aligned}
a_h(u-u_h,w_h)=\sum_{T\in\mathcal{T}_h}\int_T\beta\nabla u\cdot\nabla w_hdx-\int_\Omega fw_hdx=\sum_{e\in\mathcal{E}_h}\int_{e}\beta\nabla u\cdot\textbf{n}_e [w_h]_{e}ds,
\end{aligned}
\end{equation*}
where the jump $[w_h]_e\textbf{n}_e$ across an edge $e$ is defined as follows. 
Let $e$ be an interior edge shared by two elements $T_1^e$ and $T_2^e$, and $\textbf{n}_e$ the unit normal of $e$ pointing towards the outside of  $T_1^e$. Define 
\begin{equation*}
[w_h]_e\textbf{n}_e=(w_h|_{T_1^e}-w_h|_{T_2^e})\textbf{n}_e \qquad \mbox{ on } e.
\end{equation*}
If $e$ is an edge on the boundary of $\Omega$, then define $[w_h]_e\textbf{n}_e=w_h\textbf{n}_e$,
where $\textbf{n}_e$ is the unit normal of $e$ pointing towards the outside of  $\Omega$. Given an edge $e$ and an element $T$,  define the standard $L^2$ projection operators $P_0^e$ and $P_0^T$ as
\begin{equation*}
P_0^e f=|e|^{-1}\int_efds,\quad P_0^T f=|T|^{-1}\int_Tfdx.
\end{equation*}
It follows from the fact $\int_e[w_h]_eds=0$ that  
\begin{equation}\label{H1_2}
\begin{aligned}
\left|a_h(u-u_h,w_h)\right|=&\left|\sum_{e\in\mathcal{E}_h}\int_{e}\beta\nabla u\cdot\textbf{n}_e [w_h]_{e}ds\right|=\left|\sum_{e\in\mathcal{E}_h}\int_{e}\left(\beta\nabla u\cdot\textbf{n}_e-P_0^e(\beta\nabla u\cdot\textbf{n}_e)\right) [w_h]_{e}ds\right|\\
&\leq \left(\sum_{e\in\mathcal{E}_h}\left\|\beta\nabla  u\cdot  {\rm\mathbf{n}}_e-P_0^e(\beta\nabla  u\cdot  {\rm\mathbf{n}}_e)\right\|^2_{L^2(e)}\right)^{1/2}\left(\sum_{e\in\mathcal{E}_h}\left\| [w_h]_{e}\right \|^2_{L^2(e)} \right)^{1/2}.
\end{aligned}
\end{equation}

Let $\mathcal{T}_h^e=\{T\in\mathcal{T}_h : e\subset\partial T\}$ for all $e\in\mathcal{E}_h$. For any $w_h\in V_h^{\rm IFE}$, using the fact $w_h|_T\in H^1(T)$, we have (see \cite{brenner2008mathematical})
\begin{equation}\label{lem_jumpe}
\left\| [w_h]_{e}\right \|^2_{L^2(e)}\leq C|e|\sum_{T\in \mathcal{T}_h^e}|w_h|^2_{H^1(T)}\quad \forall e\in \mathcal{E}_h.
\end{equation}
Thus, the following estimate holds true:
\begin{equation}\label{H1_3}
\sum_{e\in\mathcal{E}_h}\left\| [w_h]_{e}\right \|^2_{L^2(e)}\leq Ch\|w_h\|^2_{a_h}.
\end{equation}

The next step is to estimate $\left\|\beta\nabla  u\cdot  {\rm\mathbf{n}}_e-P_0^e(\beta\nabla  u\cdot  {\rm\mathbf{n}}_e)\right\|_{L^2(e)}$ in (\ref{H1_2}).  Let $T$ be an element such that $e\subset \partial T$.  
If $e\in\mathcal{E}_h^{non}$ and $T\in\mathcal{T}_h^{non}$, we have the standard estimate
\begin{equation}\label{un1}
\left\|\beta\nabla  u\cdot  {\rm\mathbf{n}}_e-P_0^e(\beta\nabla  u\cdot  {\rm\mathbf{n}}_e)\right\|_{L^2(e)}\leq Ch_T^{1/2}|u|_{H^2(T)}.
\end{equation}
If $e\in\mathcal{E}_h^{non}$ and $T\in\mathcal{T}_h^{\Gamma}$,  the term can be estimated by using the fact that $e\in\Omega^s$, $s=+$ or $-$,
\begin{equation}\label{un2}
\begin{aligned}
&\left\|\beta\nabla  u\cdot  {\rm\mathbf{n}}_e-P_0^e(\beta\nabla  u\cdot  {\rm\mathbf{n}}_e)\right\|_{L^2(e)}= \left\|\beta^s\nabla  u_E^s\cdot  {\rm\mathbf{n}}_e-P_0^e(\beta^s\nabla  u_E^s\cdot  {\rm\mathbf{n}}_e)\right\|_{L^2(e)}\\
&\qquad\qquad\leq Ch_T^{-1/2}\|\beta^s\nabla u_E^s\cdot {\rm\mathbf{n}}_e -P_0^T(\beta^s\nabla  u_E^s\cdot  {\rm\mathbf{n}}_e)\|_{L^2(T)}+Ch_T^{1/2}|\beta\nabla u\cdot {\rm\mathbf{n}}_e |_{H^1(T)}\\
&\qquad\qquad\leq Ch_T^{1/2}|u_E^s|_{H^2(T)}.
\end{aligned}
\end{equation}
Hence, it follows from (\ref{un1})-(\ref{un2})  and the extension result (\ref{extension}) that 
\begin{equation}\label{H1_4}
\sum_{e\in\mathcal{E}_h^{non}}\left\|\beta\nabla  u\cdot  {\rm\mathbf{n}}_e-P_0^e(\beta\nabla  u\cdot  {\rm\mathbf{n}}_e)\right\|^2_{L^2(e)}\leq Ch\sum_{i=+,-}|u_E^s|^2_{H^2(\Omega)}\leq Ch|u|^2_{H^2(\Omega^+\cup \Omega^-)}. 
\end{equation}

For interface edges $e\in\mathcal{E}_h^\Gamma$, we cannot conclude the optimal estimate since  $(\beta\nabla u\cdot \textbf{n}_e)|_e$ may have a jump across $e\cap \Gamma$.  
Noticing that $[u]_\Gamma=0$ implies $[\nabla u \cdot\textbf{t}]_\Gamma=0$,  the jump can be derived as  
\begin{equation}\label{edge_jump}
[\beta\nabla u\cdot \textbf{n}_e]_{\Gamma}=[\beta\nabla u\cdot \textbf{n}]_{\Gamma}\textbf{n}\cdot\textbf{n}_e+[\beta\nabla u\cdot \textbf{t}]_{\Gamma}\textbf{t}\cdot\textbf{n}_e=[\beta]_\Gamma (\nabla u\cdot \textbf{t}) ( \textbf{t} \cdot\textbf{n}_e),
\end{equation}
where we have used 
$$[\beta\nabla u\cdot \textbf{n}]_\Gamma=0~\mbox{ and }~[\beta\nabla u\cdot\textbf{t}]_\Gamma=\frac{1}{2}(\beta^++\beta^-)[\nabla u\cdot\textbf{t}]_\Gamma+[\beta]_\Gamma\nabla u\cdot \textbf{t}=[\beta]_\Gamma\nabla u \cdot\textbf{t}.$$ 
The following lemma gives an estimate on interface edges.

\begin{lemma}\label{lem_edge_est} Let $u$ be the solution of  (\ref{weakform}). Assume the triangulation near the interface is quasi-uniform, i.e., there exists a constant $c$ such that $h_T\geq ch_\Gamma$ for all $T\in\mathcal{T}_h^\Gamma$. Then there holds  
\begin{equation}\label{H1_5}
\sum_{e\in\mathcal{E}_h^\Gamma}\left\|\beta\nabla  u\cdot  {\rm\mathbf{n}}_e-P_0^e(\beta\nabla  u\cdot  {\rm\mathbf{n}}_e)\right\|^2_{L^2(e)}\leq Ch_\Gamma\sum_{T\in\mathcal{T}_h^\Gamma} |u|^2_{H^2(T^+\cup T^-)} +C\left\|[\beta]_\Gamma \nabla u\cdot {\rm\mathbf{t}}\right\|^2_{H^{1/2}(\Gamma)}.
\end{equation}
\end{lemma}
\begin{proof}
Define a function $z|_{\Omega^s}=z^s$, $s=+,-$, such that 
\begin{equation*}
\begin{aligned}
&-\Delta z^s+z^s=0 \quad \mbox{ in } \Omega^s,\quad s=+,-,\\
&z^s=[\beta]_\Gamma\nabla u \cdot\textbf{t} ~\mbox{ on }~ \Gamma,~~\quad \frac{\partial z^+}{\partial \nu}=0~\mbox{ on }~\partial \Omega,~~s=+,-,
\end{aligned}
\end{equation*}
where $\nu$ is the outward unit normal vector to $\partial \Omega$. 
Since $[\beta]_\Gamma\nabla u \cdot\textbf{t}\in H^{1/2}(\Gamma)$, the function $z$ exists and satisfies
\begin{equation}\label{edge_ext}
z|_{\Gamma}=[\beta]_\Gamma\nabla u \cdot\textbf{t}~\mbox{ and }~\|z\|_{H^1(\Omega)}\leq C\|[\beta]_\Gamma\nabla u \cdot\textbf{t}\|_{H^{1/2}(\Gamma)}.
\end{equation}
For an edge $e\in\mathcal{E}_h^\Gamma$,  let $w=\beta\nabla  u\cdot  {\rm\mathbf{n}}_e$ and $T$ be an element that has $e$ as one of its edges.  Define 
\begin{equation*}
\widetilde{w}(x)=z(x)\textbf{t}(x)\cdot\textbf{n}_e~~~ \forall x~\in T \quad \mbox{ and  }  \quad
\hat{w}(x)=\left\{
\begin{aligned}
&\widetilde{w}\qquad&&\mbox{ if } x\in T^+,\\
&0&&\mbox{ if } x\in T^-.
\end{aligned}
\right.
\end{equation*}
From (\ref{edge_jump}) and (\ref{edge_ext}), we have $[w-\hat{w}]_{\Gamma\cap T}=0$. Thus, $w-\hat{w}\in H^1(T)$. By the property of the $L^2$ projection operator $P_0^e$ and the standard trace inequality, we  infer
\begin{equation*}
\begin{aligned}
\|w-P_0^e(w)\|^2_{L^2(e)}&\leq \|w-P_0^T(w-\hat{w})\|^2_{L^2(e)}=\|w-\hat{w}+\hat{w}-P_0^T(w-\hat{w})\|^2_{L^2(e)}\\
&\leq 2\|w-\hat{w}-P_0^T(w-\hat{w})\|^2_{L^2(e)}+2\|\widetilde{w}\|^2_{L^2(e)}\\
&\leq C(h_T^{-1}\|w-\hat{w}-P_0^T(w-\hat{w}) \|^2_{L^2(T)}+h_T| w-\hat{w}|^2_{H^1(T)})+2\|z\|^2_{L^2(e)}\\
&\leq C(h_T| w-\hat{w}|^2_{H^1(T)}+h_T^{-1}\|z\|^2_{L^2(T)}+h_T|z|^2_{H^1(T)})\\
&\leq C(h_T| w|^2_{H^1(T^+\cup T^-)}+h_T^{-1}\|z\|^2_{L^2(T)}+h_T|z|^2_{H^1(T)}).
\end{aligned}
\end{equation*}
Summing over all interface edges and using Lemma~\ref{strip}, we get
\begin{equation*}
\begin{aligned}
\sum_{e\in\mathcal{E}_h^\Gamma}&\left\|\beta\nabla  u\cdot  {\rm\mathbf{n}}_e-P_0^e(\beta\nabla  u\cdot  {\rm\mathbf{n}}_e)\right\|^2_{L^2(e)}\leq \sum_{T\in\mathcal{T}_h^\Gamma} C(h_T| w|^2_{H^1(T^+\cup T^-)}+h_T^{-1}\|z\|^2_{L^2(T)}+h_T|z|^2_{H^1(T)})\\
&\qquad\leq Ch_\Gamma\sum_{T\in\mathcal{T}_h^\Gamma}|u|^2_{H^2(T^+\cup T^-)}+Ch_\Gamma^{-1}\|z\|^2_{L^2(U(\Gamma,h_\Gamma))}+Ch_\Gamma|z|^2_{H^1(\Omega)}\\
&\qquad\leq Ch_\Gamma\sum_{T\in\mathcal{T}_h^\Gamma}|u|^2_{H^2(T^+\cup T^-)}+C\|z\|^2_{H^1(\Omega)},
\end{aligned}
\end{equation*}
which together with (\ref{edge_ext}) yields this lemma.
\end{proof}

\begin{remark}\label{remark_suboptimal}
\begin{figure} [htbp]
\centering
\includegraphics[height=0.35\textwidth]{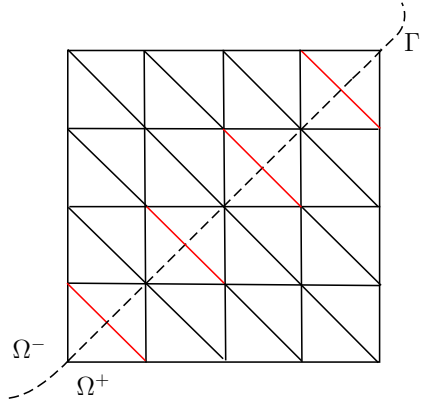}
 \caption{A triangulation of a region $\Lambda$ contained in $\Omega$; red: interface edges.}\label{interface_edge} 
\end{figure}
The estimate (\ref{H1_5}) is sharp, i.e., we cannot find a better upper bound for the approximation error than $O(1)$ when $ [\beta]_\Gamma \nabla u\cdot {\rm\mathbf{t}} \not=0$ on $\Gamma$. We explain it by a concrete example as illustrated in Figure~\ref{interface_edge}. The domain $\Omega$ contains a region $\Lambda$ such that $\Lambda=(0,1)^2$, $\Lambda^+=\{x=(x_1,x_2)\in\Lambda : x_1>x_2\}$, $\Lambda^-=\{x=(x_1,x_2)\in\Lambda : x_1<x_2\}$. The interface contained in the region is $\Gamma\cap \Lambda=\{x=(x_1,x_2)\in\Lambda : x_1=x_2\}$. We use uniform triangulations as shown in Figure~\ref{interface_edge} and only consider these interface edges which are contained in the region $\Lambda$. Obviously, $\textbf{n}_e=\textbf{t}=(\frac{1}{\sqrt{2}},\frac{1}{\sqrt{2}})^T$ and $\textbf{n}=(\frac{1}{\sqrt{2}},-\frac{1}{\sqrt{2}})^T$. Let $\beta^+=2$, $\beta^-=1$ and the exact solution $u(x_1,x_2)$  be a piecewise linear function on $\Lambda^+$, $\Lambda^-$ such that
\begin{equation*}
\beta^+\nabla u^+\cdot \textbf{n}=\beta^-\nabla u^-\cdot \textbf{n}=1 \quad \mbox{ and }\quad u|_{\Gamma\cap\Lambda}=\frac{1}{\sqrt{2}}(x_1+x_2).
\end{equation*}
Thus, $\nabla u\cdot\textbf{t}=1$ and $(\beta\nabla u\cdot\textbf{n}_e)|_{e^+}=2$, $(\beta\nabla u\cdot\textbf{n}_e)|_{e^-}=1$ for all $e\subset\Lambda$ and $e\in\mathcal{E}_h^\Gamma$. Therefore,
\begin{equation*}
\begin{aligned}
&\left\|\beta\nabla  u\cdot  {\rm\mathbf{n}}_e-P_0^e(\beta\nabla  u\cdot  {\rm\mathbf{n}}_e)\right\|^2_{L^2(e)}= 
\inf_{c_e\in\mathbb{R}}\left\|\beta\nabla  u\cdot  {\rm\mathbf{n}}_e-c_e\right\|^2_{L^2(e)}\\
&\qquad=\inf_{c_e\in\mathbb{R}}\left(\frac{|e|}{2}(2-c_e)^2+\frac{|e|}{2}(1-c_e)^2\right)\geq \frac{|e|}{4}\geq Ch.
\end{aligned}
\end{equation*}
Using the fact that the number of interface edges contained in $\Lambda$ is $O(h^{-1})$, we  observe that
\begin{equation*}
\sum_{e\in\mathcal{E}_h^\Gamma}\left\|\beta\nabla  u\cdot  {\rm\mathbf{n}}_e-P_0^e(\beta\nabla  u\cdot  {\rm\mathbf{n}}_e)\right\|^2_{L^2(e)}\geq \sum_{e\in\mathcal{E}_h^\Gamma,e\subset\Lambda} Ch\geq C.
\end{equation*}
\end{remark}

With the above discussions, we  have the following error estimate for the existing nonconforming IFE method. 
\begin{theorem}\label{the_H1_err}
Let $u$ and $u_h$ be the solutions of (\ref{weakform}) and (\ref{method1}), respectively. Under the assumption of  Lemma~\ref{lem_edge_est},   the following discretization error estimate holds true:
\begin{equation}\label{H1_error}
\|u-u_h\|_{a_h}\leq Ch\|u\|_{H^2(\Omega^+\cup\Omega^-)}+Ch^{1/2}\left\|[\beta]_\Gamma \nabla u\cdot {\rm\mathbf{t}}\right\|_{H^{1/2}(\Gamma)}.
\end{equation}
\end{theorem}

\begin{proof} 
It follows from  (\ref{H1_4}) and (\ref{H1_5}) that 
\begin{equation}\label{edge_est}
\sum_{e\in\mathcal{E}_h}\left\|\beta\nabla  u\cdot  {\rm\mathbf{n}}_e-P_0^e(\beta\nabla  u\cdot  {\rm\mathbf{n}}_e)\right\|^2_{L^2(e)}\leq Ch |u|^2_{H^2(\Omega^+\cup \Omega^-)} +C\left\|[\beta]_\Gamma \nabla u\cdot {\rm\mathbf{t}}\right\|^2_{H^{1/2}(\Gamma)},
\end{equation}
which together with (\ref{H1_2}) and (\ref{H1_3}) yields
\begin{equation}\label{incons}
\left|a_h(u-u_h,w_h)\right|\leq Ch^{1/2}\|w_h\|_{a_h}\left(h^{1/2}|u|_{H^2(\Omega^+\cup \Omega^-)} +\left\|[\beta]_\Gamma \nabla u\cdot {\rm\mathbf{t}}\right\|_{H^{1/2}(\Gamma)}\right).
\end{equation}
Combining (\ref{strang}), (\ref{H1_est}) and  (\ref{incons}), we complete the proof.
\end{proof}

Theorem~\ref{the_H1_err} suggests that the solution of the nonconforming IFE method (\ref{method1}) only converges with  a suboptimal convergence rate $O(h^{1/2})$ in the energy norm if $[\beta]_\Gamma \nabla u\cdot \textbf{t}\not=0$ on $\Gamma$. 
Applying global trace inequalities on $\Omega^+$ and $\Omega^-$ to the second term on the right-hand side of  (\ref{H1_error}), we get the following corollary.

\begin{corollary}
Let $u$ and $u_h$ be the solutions of (\ref{weakform}) and (\ref{method1}), respectively. Suppose that $[\beta]_\Gamma\not=0$ and $\nabla u\cdot \textbf{t}\not=0$ on $\Gamma$. Under the assumption of  Lemma~\ref{lem_edge_est}, we have
\begin{equation*}
\|u-u_h\|_{a_h}\leq Ch^{1/2}\|u\|_{H^2(\Omega^+\cup\Omega^-)}.
\end{equation*}
\end{corollary}

\begin{remark}
 Using the standard duality argument (see \cite[p. 284]{brenner2008mathematical}), we can also derive the following suboptimal  $L^2$ error estimate:
\begin{equation*}
\|u-u_h\|_{L^2(\Omega)}\leq Ch\|u\|_{H^2(\Omega^+\cup\Omega^-)}.
\end{equation*}
\end{remark}

\section{A new nonconforming IFE method and error estimates}\label{sec_method2}
To recover the optimal convergence rates, different from the partially penalized IFE method in \cite{zhangphd}, we propose a parameter-free nonconforming IFE method.
 On each interface element $T\in\mathcal{T}_h^\Gamma$, define
\begin{equation}\label{lift1}
W_h(T):=\{w_h\in (L^2(T))^2: w_h=\nabla v_h~~~\forall v_h\in S_h(T) \}.
\end{equation}
We also define a space associated with an edge $e\in\mathcal{E}_h^\Gamma$ as
\begin{equation}\label{lift2}
W_e:=\{w_h\in (L^2(\Omega))^2:~ w_h|_{T_1^e}\in W_h(T^e_1),~ w_h|_{T^e_2}\in W_h(T^e_2),~ w_h|_{\Omega\backslash(T^e_1\cup T^e_2)}=0\},
\end{equation}
where  $T_1^e$ and $T_2^e$ are elements sharing the common edge $e$.
For simplicity of the implementation, the coefficient $\beta(x)$ is approximated  by 
\begin{equation*}
\beta_h(x)=\left\{
\begin{aligned}
&\beta^+(x)~\mbox{ if }~x\in\Omega^+_h,\\
&\beta^-(x)~\mbox{ if }~x\in\Omega^-_h.
\end{aligned}\right.
\end{equation*}
Define a local  {\em  lifting operator}  $r_e : L^2(e)\rightarrow W_e$ such that 
\begin{equation}\label{def_lift}
\int_\Omega \beta_h(x) r_e(\varphi)\cdot w_hdx=\int_e\{\beta_h w_h\cdot \textbf{n}_e\}_e\, \varphi \,ds\qquad \forall w_h\in W_e,
\end{equation}
where $\{v\}_e=\frac{1}{2}(v|_{T^e_1}+v|_{T^e_2})$. Obviously, $r_e(\varphi)$ exists uniquely for any $\varphi\in L^2(e)$ and  the support of $r_e(\varphi)$ is $T_1^e\cup T_2^e$. We emphasize that  the cost in computing $r_e(\varphi)$ with given $\varphi$ is not significant in general in practical implementation because the dimension of the space $W_h(T)$ is 2 for the Crouzeix-Raviart element and 3 for the rotated-$Q_1$ element. We also note that  for the Crouzeix-Raviart element we can express $r_e(\varphi)$ explicitly if we choose orthogonal basis functions of the space $W_h(T)$ as
\begin{equation*}
\omega_1(x)=\textbf{t}_{h},
\qquad
\omega_2(x)=
\left\{
\begin{aligned}
&\beta^-\textbf{n}_{h}~~~~\mbox{ if } x\in T_{h}^+\\
&\beta^+\textbf{n}_{h}~~~~\mbox{ if } x\in T_{h}^-\\
\end{aligned}\right..
\end{equation*}

The new nonconforming IFE method is to find $u_h\in V_{h,0}^{\rm IFE}$ such that
\begin{equation}\label{method2}
A_h(u_h,v_h):=\widetilde{a}_h(u_h,v_h)+b_h(u_h,v_h)+s_h(u_h,v_h)=\int_{\Omega}fv_hdx \qquad\forall v_h\in V_{h,0}^{\rm IFE},
\end{equation}
where
\begin{equation}\label{ph2}
\begin{aligned}
&\widetilde{a}_h(u_h,v_h)=\sum_{T\in\mathcal{T}_h}\int_T\beta_h(x)\nabla u_h\cdot\nabla v_hdx,\\
&b_h(u_h,v_h)=-\sum_{e\in\mathcal{E}_h^\Gamma}\int_e\left(\{\beta_h\nabla u_h\cdot\textbf{n}_e\}_e[v_h]_e+\{\beta_h\nabla v_h\cdot\textbf{n}_e\}_e[u_h]_e\right)ds,\\
&s_h(u_h,v_h)=4\sum_{e\in\mathcal{E}_h^\Gamma}\int_{T_1^e\cup T_2^e}\beta_h(x)r_e([u_h]_e)\cdot r_e([v_h]_e)dx.
\end{aligned}
\end{equation}

Clearly, the new nonconforming IFE method is symmetric and parameter-free. 
In our recent work  \cite{2021ji_IFE}, we studied a similar parameter-free IFE method using nodal values as degrees of freedom. However, the method has a limitation, i.e., the maximum angle of the interface elements should be less than or equal to $\pi/2$. In contrast, the nonconforming IFE method proposed in this paper works without this restriction. The IFE basis functions using integral-value degrees of freedom are unisolvent on arbitrary triangles which is a significant advantage over the IFE method using nodal values as degrees of freedom.

Now we analyze the new method. 
For all $v\in H_0^1(\Omega)\cap\widetilde{H}^2(\Omega)+ V_{h,0}^{\rm IFE}$, define the following mesh-dependent norm
\begin{equation*}
\|v\|_{\widetilde{a}_h}:=\sqrt{\widetilde{a}_h(v,v)}
\end{equation*}
and
\begin{equation}\label{def_nor2}
\interleave v \interleave_h:=\left(\|v\|_{\widetilde{a}_h}^2+\sum_{e\in\mathcal{E}_h^\Gamma}|e|\|\{\beta_h\nabla v\}_e\|^2_{L^2(e)}+\sum_{e\in\mathcal{E}_h^\Gamma}|e|^{-1}\| [v]_e\|^2_{L^2(e)}+s_h(v,v)\right)^{1/2}.
\end{equation}
The continuity of the bilinear form $A_h(\cdot,\cdot)$ is verified directly from the Cauchy-Schwarz inequality
\begin{equation}\label{conti}
|A_h(w,v)|\leq \interleave w \interleave_h \interleave v\interleave_h\qquad\forall w,v\in H_0^1(\Omega)\cap\widetilde{H}^2(\Omega)+ V_{h,0}^{\rm IFE}.
\end{equation}
The following lemma demonstrates the coercivity of the bilinear form $A_h(\cdot,\cdot)$  on the IFE space $V_{h,0}^{\rm IFE}$ with respect to the norm $\|\cdot\|_{\widetilde{a}_h}$.
\begin{lemma}\label{lem_Coercivity}
We have
\begin{equation}\label{Coercivity}
A_h(v_h,v_h)\geq \frac{1}{2}\|v_h\|_{\widetilde{a}_h}^2\qquad \forall v_h\in V_{h,0}^{\rm IFE}.
\end{equation}
\end{lemma}
\begin{proof}
For all $v_h\in V_{h,0}^{\rm IFE}$, choosing $w_h|_{T_{1}^e\cup T_{2}^e}=\nabla v_h$, $w_h|_{\Omega\backslash(T_{1}^e\cup T_{2}^e) }=0$ in (\ref{def_lift}) and using the fact that the support of  $r_e(\varphi)$ is $T_{1}^e\cup T_{2}^e$,  we have
\begin{equation*}
\int_{T_{1}^e\cup T_{2}^e} \beta_h(x) r_e(\varphi)\cdot \nabla v_h dx=\int_\Omega \beta_h(x) r_e(\varphi)\cdot \nabla v_h dx=\int_e\{\beta_h \nabla v_h\cdot \textbf{n}_e\}_e\, \varphi \,ds.
\end{equation*}
It follows from the Cauchy-Schwarz inequality that 
\begin{equation}\label{pro_coer1}
\begin{aligned}
&|b_h(v_h,v_h)|=\left|2\sum_{e\in\mathcal{E}_h^\Gamma}\int_e\{\beta_h\nabla v_h\cdot\textbf{n}_e\}_e[v_h]_eds\right|\\
&\leq \left(4\sum_{e\in\mathcal{E}_h^\Gamma} \int_{\Omega} \beta_h r_e([v_h]_e)\cdot r_e([v_h]_e)dx\right)^{1/2}\left(\sum_{e\in\mathcal{E}_h^\Gamma} \int_{T_{1}^e\cup T_{2}^e} \beta_h \nabla v_h \cdot \nabla v_h dx\right)^{1/2}.
\end{aligned}
\end{equation}
Since each interface element has at most two interface edges from Assumption {\ref{assum_2}}, each element is calculated at most twice. Therefore,
\begin{equation}\label{pro_coer2}
\sum_{e\in\mathcal{E}_h^\Gamma} \int_{T_{1}^e\cup T_{2}^e} \beta_h \nabla v_h \cdot \nabla v_h dx\leq 2\sum_{T\in\mathcal{T}_h}\int_T \beta_h(x) \nabla v_h \cdot \nabla v_h dx.
\end{equation}
Combining (\ref{ph2}), (\ref{pro_coer1}) and (\ref{pro_coer2}), we have
\begin{equation*}
\begin{aligned}
\left|b_h(v_h,v_h)\right|&\leq \left(s_h(v_h,v_h)\right)^{1/2}\left(2\sum_{T\in\mathcal{T}_h}\int_T \beta_h(x) \nabla v_h \cdot \nabla v_h dx\right)^{1/2}\\
&\leq s_h(v_h,v_h)+\frac{1}{2} \sum_{T\in\mathcal{T}_h}\int_T \beta_h(x) \nabla v_h \cdot \nabla v_h dx,
\end{aligned}
\end{equation*}
 which together with (\ref{ph2}) yields the result 
\begin{equation*}
\begin{aligned}
A_h(v_h,v_h)&=\widetilde{a}_h(v_h,v_h)+b_h(v_h,v_h)+s_h(v_h,v_h)\\
&\geq \frac{1}{2} \sum_{T\in\mathcal{T}_h}\int_T \beta_h(x) \nabla v_h \cdot \nabla v_h dx = \frac{1}{2}\|v_h\|_{\widetilde{a}_h}^2.
\end{aligned}
\end{equation*}
\end{proof}
Next, we show the equivalence of the $\|\cdot\|_{\widetilde{a}_h}$-norm and the $ \interleave\cdot\interleave_h$-norm on the IFE space $V_{h,0}^{\rm IFE}$. To begin with, we need the following trace inequality for the IFE shape functions in $S_h(T)$ which can be verified via straightforward calculations. We also refer readers to \cite[Theorem 2.7]{2019Non_Lin}. 
\begin{lemma}
There exists a  constant $C$ independent of $h$ and the interface location relative to the mesh such that
\begin{equation}\label{trace_IFE}
\|\nabla v_h\|_{L^2(\partial T)}\leq Ch_T^{-1/2}\|\nabla v_h\|_{L^2(T)}~\quad \forall v_h\in S_h(T)\quad\forall T\in\mathcal{T}_h^\Gamma.
\end{equation}
\end{lemma}
We also need the following stability  estimate for the local lifting operator $r_e$.
\begin{lemma}\label{lem_stab_lift}
There exists a  constant $C$ independent of $h$ and the interface location relative to the mesh such that
\begin{equation*}
\|r_e(\varphi)\|_{L^2(\Omega)}\leq C|e|^{-1/2}\|\varphi\|_{L^2(e)}\quad \forall \varphi\in L^2(e)\quad \forall e\in\mathcal{E}_h^\Gamma.
\end{equation*}
\end{lemma}
\begin{proof}
Since the support of $r_e(\varphi)$ is $T_1^e\cup T_2^e$,
choosing  $w_h=r_e(\varphi)$ in (\ref{def_lift}) yields 
\begin{equation}\label{pro_lem_lif1}
\begin{aligned}
\|r_e(\varphi)\|^2_{L^2(\Omega)}&\leq C\|\beta_h^{1/2}r_e(\varphi)\|^2_{L^2(T_1^e\cup T_2^e)}=C\int_e\{\beta_h r_e(\varphi)\cdot\textbf{n}_e\}_e\varphi ds\\
&\leq C\|\{\beta_hr_e(\varphi)\}_e\|_{L^2(e)}\|\varphi\|_{L^2(e)}\leq C\|\varphi\|_{L^2(e)}\sum_{i=1,2}\|r_e(\varphi)|_{T_i^e}\|_{L^2(e)}.
\end{aligned}
\end{equation}
Note $r_e(\varphi)|_{T_1^e}\in W_h(T_1^e)$, we know from (\ref{lift1}) that there exists a function $v_h\in S_h(T_1^e)$ such that $r_e(\varphi)|_{T_1^e}=\nabla v_h$. By the trace inequality (\ref{trace_IFE}) for the IFE shape functions, we have
\begin{equation*}
  \|r_e(\varphi)|_{T_1^e} \|_{L^2(e)}=\|\nabla v_h\|_{L^2(e)}\leq Ch_T^{-1/2}\|\nabla v_h\|_{L^2(T_1^e)}=Ch_T^{-1/2}\|r_e(\varphi)\|_{L^2(T_1^e)},
\end{equation*}
which, together with (\ref{pro_lem_lif1}) and a similar estimate on $T_2^e$, completes the proof of this lemma.
\end{proof}

\begin{remark}\label{remark_PPIFEM}
Lemma~\ref{lem_stab_lift} indicates that  the parameter-free stabilization is weaker than the standard $L^2$ stabilization. Thus, our analysis is also valid for the nonconforming PPIFE method in \cite{zhangphd}, i.e.,  replacing $s_h$ in (5.4)  by $\tilde{s}_h$ defined by 
$$\tilde{s}_h(u_h,v_h)=\sum_{e\in\mathcal{E}_h^\Gamma}\frac{\eta_e}{|e|}\int_{e}[u_h]_e [v_h]_edx,$$
where $\eta_e>0$ should be sufficiently large.
\end{remark}

We now prove the norm-equivalence in the following lemma.
\begin{lemma}\label{lema_equ}
There exists a constant $C$ independent of $h$ and the interface location relative to the mesh such that
\begin{equation}\label{equali}
\|v_h\|_{\widetilde{a}_h} \leq \interleave v_h \interleave_h\leq C\|v_h\|_{\widetilde{a}_h}\qquad \forall v_h\in V_{h,0}^{\rm IFE}.
\end{equation}
\end{lemma}

\begin{proof}
We just need to prove the second inequality since the first inequality is obvious. By the trace inequality (\ref{trace_IFE}) for the IFE shape functions, we can see that
\begin{equation}\label{pro_equ1}
\begin{aligned}
\sum_{e\in\mathcal{E}_h^\Gamma}|e|&\|\{\beta_h\nabla v_h\}_e\|^2_{L^2(e)}\leq C\sum_{e\in\mathcal{E}_h^\Gamma} \sum_{T\in\mathcal{T}_h^e}\|\nabla v_h\|^2_{L^2(T)}\leq C\|v_h\|^2_{\widetilde{a}_h}.
\end{aligned}
\end{equation}
From (\ref{lem_jumpe}), we have 
\begin{equation}\label{pro_equ2}
\sum_{e\in\mathcal{E}_h^\Gamma}|e|^{-1}\| [v_h]_e\|^2_{L^2(e)}\leq C\sum_{e\in\mathcal{E}_h^\Gamma} \sum_{T\in\mathcal{T}_h^e}\|\nabla v_h\|^2_{L^2(T)}\leq C\|v_h\|^2_{\widetilde{a}_h},
\end{equation}
which, together with  Lemma~\ref{lem_stab_lift} for the local lifting operator,  leads to
\begin{equation}\label{pro_equ3}
s_h(v_h,v_h)\leq C\sum_{e\in\mathcal{E}_h^\Gamma}\left\|r_e([v_h]_e)\right\|^2_{L^2(\Omega)}\leq C\sum_{e\in\mathcal{E}_h^\Gamma}|e|^{-1} \|[v_h]_e\|^2_{L^2(e)}\leq C\|v_h\|^2_{\widetilde{a}_h}.
\end{equation}
Combining (\ref{def_nor2}), (\ref{pro_equ1})-(\ref{pro_equ3}), we get  the second inequality in (\ref{equali}).
\end{proof}

The following lemma provides an optimal estimate for the interpolation error in terms of the norm $\interleave\cdot \interleave_h$.
\begin{lemma}\label{ener_app}
Suppose $v\in \widetilde{H}^2(\Omega)$, then there exists a constant $C$ independent of $h$ and the interface location relative to the mesh such that
\begin{equation*}
\interleave v-I_h^{\rm IFE}v \interleave_h\leq Ch\|v\|_{H^2(\Omega^+\cup\Omega^-)}.
\end{equation*}
\end{lemma}
\begin{proof}
The first term in the norm $\interleave\cdot\interleave_h$ can be bounded by Theorem~\ref{theorem_interpolation},
\begin{equation}\label{pro_int11}
\|v-I_h^{\rm IFE}v\|_{\widetilde{a}_h}\leq Ch\|v\|_{H^2(\Omega^+\cup\Omega^-)}.
\end{equation}
Since $(v-I_h^{\rm IFE}v)|_{T}\in H^1(T)$ for all $T\in\mathcal{T}_h^\Gamma$, by the standard trace inequality and Theorem~\ref{chazhi_error}, we have
\begin{equation}\label{pro_int1b}
\begin{aligned}
&\sum_{e\in\mathcal{E}_h^\Gamma}|e|^{-1}\| [v-I_h^{\rm IFE}v]_e\|^2_{L^2(e)}\leq C\sum_{e\in\mathcal{E}_h^\Gamma}\sum_{T
\in\mathcal{T}_h^e}\left(h_T^{-2}\|v-I_h^{\rm IFE}v\|^2_{L^2(T)}+|v-I_h^{\rm IFE}v|^2_{H^1(T)}\right)\\
&\qquad\leq C\sum_{T\in\mathcal{T}_h^\Gamma}\left(h_T^{-2}\|v-I_h^{\rm IFE}v\|^2_{L^2(T)}+|v-I_h^{\rm IFE}v|^2_{H^1(T)}\right)\leq Ch_\Gamma^2\|v\|^2_{H^2(\Omega^+\cup\Omega^-)},
\end{aligned}
\end{equation}
which, together with Lemma~\ref{lem_stab_lift}, implies 
\begin{equation}\label{pro_int3b}
\begin{aligned}
s_h(v-I_h^{\rm IFE}v,v-I_h^{\rm IFE}v)\leq C\sum_{e\in\mathcal{E}_h^\Gamma}|e|^{-1}\|[v-I_h^{\rm IFE}v]_e\|^2_{L^2(e)}\leq Ch_\Gamma^2\|v\|^2_{H^2(\Omega^+\cup\Omega^-)}.
\end{aligned}
\end{equation}
Let $e^s=e\cap \Omega^s$, $s=+,-$. Recalling the notations in (\ref{def_pm})-(\ref{extension}), we have
\begin{equation*}
\begin{aligned}
\|\{\beta_h\nabla (v&-I_h^{\rm IFE}v)\}_e\|^2_{L^2(e)}=\|\{\beta_h\nabla (v-I_h^{\rm IFE}v)\}_e\|^2_{L^2(e^+)}+\|\{\beta_h\nabla (v-I_h^{\rm IFE}v)\}_e\|^2_{L^2(e^-)}\\
&\leq C\|\{\nabla (v_E^+-(I_h^{\rm IFE}v)^+)\}_e\|^2_{L^2(e)}+C\|\{\nabla (v_E^--(I_h^{\rm IFE}v)^-)\}_e\|^2_{L^2(e)}.
\end{aligned}
\end{equation*}
Then using the standard trace inequality and Theorem~\ref{chazhi_error},  we infer 
\begin{equation}\label{pro_int2b}
\begin{aligned}
&\sum_{e\in\mathcal{E}_h^\Gamma}|e|\|\{\beta_h\nabla (v-I_h^{\rm IFE}v)\}_e\|^2_{L^2(e)}\\
&\leq C\sum_{T\in\mathcal{T}_h^\Gamma}\sum_{s=+,-}\left(|v_E^s-(I_h^{\rm IFE}v)^s|^2_{H^1(T)}+h_T^2|v_E^s-(I_h^{\rm IFE}v)^s|^2_{H^2(T)}\right)\\
&\leq Ch_\Gamma^2\|v\|^2_{H^2(\Omega^+\cup\Omega^-)}.
\end{aligned}
\end{equation}
The lemma follows from  (\ref{def_nor2}), (\ref{pro_int11})-(\ref{pro_int2b}).
 \end{proof}
 
The following lemma concerns the errors caused by replacing $\beta(x)$ by $\beta_h(x)$.
\begin{lemma}\label{lem_betah}
Let $v\in\widetilde{H}^2(\Omega)$ and $w\in V_h^{\rm IFE}+H^1(\Omega)$. Then there exists a constant $C$ independent of $h$ and the interface location relative to the mesh such that
\begin{equation}\label{lem_betah1}
|a_h(v,w)-\widetilde{a}_h(v,w)|\leq Ch_\Gamma\|v\|_{H^2(\Omega^+\cup\Omega^-)}\left(\sum_{T\in\mathcal{T}_h^\Gamma}\|\nabla w\|^2_{L^2(T^\triangle)}\right)^{1/2}.
\end{equation}
Furthermore, if $w\in \widetilde{H}^2(\Omega)$, there holds
\begin{equation}\label{lem_betah2}
|a_h(v,w)-\widetilde{a}_h(v,w)|\leq Ch_\Gamma^2\|v\|_{H^2(\Omega^+\cup\Omega^-)}\|w\|_{H^2(\Omega^+\cup\Omega^-)}.
\end{equation}
\end{lemma}
 \begin{proof}
The Cauchy-Schwarz inequality gives 
 \begin{equation}\label{pr1_betah1}
 \begin{aligned}
 |a_h(v,w)-\widetilde{a}_h(v,w)|&=\left|\sum_{T\in\mathcal{T}_h^\Gamma}\int_{T^\triangle}(\beta-\beta_h)\nabla v\cdot \nabla wdx\right|\\
 &\leq C\left(\sum_{T\in\mathcal{T}_h^\Gamma} \|\nabla v\|^2_{L^2(T^\triangle)}\right)^{1/2}\left(\sum_{T\in\mathcal{T}_h^\Gamma} \|\nabla w\|^2_{L^2(T^\triangle)}\right)^{1/2}.
 \end{aligned}
 \end{equation}
Using Lemma~\ref{lem_h3} and the global trace inequalities on $\Omega^+$ and $\Omega^-$, we can see that
\begin{equation}\label{pr1_betah2}
\begin{aligned}
\sum_{T\in\mathcal{T}_h^\Gamma} \|\nabla v\|^2_{L^2(T^\triangle)}&=\sum_{T\in\mathcal{T}_h^\Gamma}\sum_{s=+,-}\|\nabla v^s\|^2_{L^2(T^\triangle\cap T^s)}\leq\sum_{T\in\mathcal{T}_h^\Gamma}\sum_{s=+,-}\|\nabla v_E^s\|^2_{L^2(T^\triangle) }\\
&\leq C\sum_{T\in\mathcal{T}_h^\Gamma}\sum_{s=+,-} \left(h_T^2\|\nabla v^s\|^2_{L^2(T\cap \Gamma)}+h_T^4|v_E^s|^2_{H^2(T^\triangle)}\right)\\
&\leq Ch_\Gamma^2\sum_{s=+,-}\| \nabla v^s \|^2_{L^2(\Gamma)}+Ch_\Gamma^4\sum_{s=+,-}| v_E^s|^2_{H^2(\Omega)}\\
&\leq Ch_\Gamma^2\sum_{s=+,-} \|v \|^2_{H^2(\Omega^s)}=Ch_\Gamma^2\|v\|^2_{H^2(\Omega^+\cup \Omega^-)}.
\end{aligned}
\end{equation}
The estimate (\ref{lem_betah1}) follows from (\ref{pr1_betah1}) and (\ref{pr1_betah2}). If $w\in \widetilde{H}^2(\Omega)$, then similar to (\ref{pr1_betah2}), 
\begin{equation}\label{pr1_betah3}
\sum_{T\in\mathcal{T}_h^\Gamma} \|\nabla w\|^2_{L^2(T^\triangle)}\leq Ch_\Gamma^2\|w\|^2_{H^2(\Omega^+\cup \Omega^-)}.
\end{equation}
The  estimate (\ref{lem_betah2}) then follows from (\ref{lem_betah1}) and (\ref{pr1_betah3}).
 \end{proof}
 
With these preparations, we are ready to derive the $H^1$ error estimate for the new nonconforming IFE method. 
 \begin{theorem}\label{theo_mainH1}
Let $u$ and $u_h$ be the solutions of  (\ref{weakform}) and (\ref{method2}), respectively.  Then there exists a constant $C$ independent of $h$ and the interface location relative to the mesh such that
\begin{equation}\label{H_1_error_method2}
\interleave u-u_h \interleave_h\leq Ch\|u\|_{H^2(\Omega^+\cup\Omega^-)}.
\end{equation}
\end{theorem}
\begin{proof}
From Lemma~\ref{lem_Coercivity} and Lemma~\ref{lema_equ}, we know that the bilinear form $A_h(\cdot,\cdot)$ is also coercive on $V_h^{\rm IFE}$ with respect to the norm $\interleave\cdot\interleave_h$. Thus,  the second Strang lemma implies
\begin{equation}\label{h1_pro1}
\interleave u-u_h\interleave_h \leq C\left\{\inf_{v_h\in V_{h,0}^{\rm IFE}}\interleave u-v_h\interleave_h+\sup_{w_h\in V_{h,0}^{\rm IFE}\backslash\{0\}}\frac{|A_h(u-u_h,w_h)|}{\interleave w_h\interleave_h}\right\}.
\end{equation}
 The first term of the right-hand side of (\ref{h1_pro1}) can be bounded by Lemma~\ref{ener_app}
\begin{equation}\label{H1_est2}
\inf_{v_h\in V_{h,0}^{\rm IFE}}\interleave u-v_h\interleave_h\leq \interleave u-I_h^{\rm IFE}u\interleave_h\leq Ch\|u\|_{H^2(\Omega^+\cup\Omega^-)}.
\end{equation}
Multiplying (\ref{p1.1}) by $w_h\in V_{h,0}^{\rm IFE}$ and using integration by parts, we obtain 
\begin{equation}\label{consis1}
\begin{aligned}
\int_{\Omega}fw_hdx&=a_h(u,w_h)+s_h(u,w_h)-\sum_{e\in\mathcal{E}_h}\int_e\{\beta\nabla u\cdot\textbf{n}_e\}_e[w_h]_e+\{\beta\nabla w_h\cdot\textbf{n}_e\}_e[u]_eds\\
&=A_h(u,w_h)-\sum_{e\in\mathcal{E}_h^{non}}\int_e\beta\nabla u\cdot\textbf{n}_e[w_h]_eds+a_h(u,w_h)-\widetilde{a}_h(u,w_h),
\end{aligned}
\end{equation}
where  $\beta_h(x)=\beta(x)$ on  edges, $[u]_e=0$, $[\beta\nabla u\cdot\textbf{n}_e]_e=0$ and $r_e([u]_e)=0$ are used.
It follows from (\ref{consis1}) and (\ref{method2}) that 
\begin{equation}\label{pro_L2_ex}
A_h(u-u_h,w_h)=\sum_{e\in\mathcal{E}_h^{non}}\int_e\beta\nabla u\cdot\textbf{n}_e[w_h]_eds+\widetilde{a}_h(u,w_h)-a_h(u,w_h).
\end{equation}
Hence, by (\ref{H1_2}), (\ref{H1_3}), (\ref{H1_4}), and Lemma~\ref{lem_betah}, we have
\begin{equation*}
|A_h(u-u_h,w_h)|\leq Ch\|u\|_{H^2(\Omega^+\cup\Omega^-)}\interleave w_h \interleave_h,
\end{equation*}
which, together with (\ref{h1_pro1}) and (\ref{H1_est2}), completes the proof of the theorem.
\end{proof}

The  optimal $L^2$ error estimate is also derived by using the standard duality argument below.
\begin{theorem}\label{theo_mainL2}
Let $u$ and $u_h$ be the solutions of (\ref{weakform}) and (\ref{method2}), respectively.  Then there exists a constant $C$ independent of $h$ and the interface location relative to the mesh such that
\begin{equation}\label{L2_err_method2}
\|u-u_h\|_{L^2(\Omega)}\leq Ch^2\|u\|_{H^2(\Omega^+\cup \Omega^-)}.
\end{equation}
\end{theorem}
\begin{proof}
Let $z\in H_0^1(\Omega)$ be the solution of the following auxiliary problem
\begin{equation}\label{auxi}
a(v,z)=\int_\Omega (u-u_h)vdx\quad\forall v\in H_0^1(\Omega).
\end{equation}
Since $u-u_h\in L^2(\Omega)$, it follows from Theorem~\ref{regular_assumption} that
\begin{equation}\label{reg_L2}
z\in \widetilde{H}^2(\Omega)~~\mbox{ and }~~\|z\|_{H^2(\Omega^+\cup\Omega^-)}\leq C\|u-u_h\|_{L^2(\Omega)}.
\end{equation}
Let $z_h\in V_{h,0}^{\rm IFE}$ be the solution of   the new nonconforming IFE method applied to the auxiliary problem~(\ref{auxi}), i.e.,
\begin{equation}\label{pro_L21}
A_h(v_h,z_h)=\int_{\Omega}(u-u_h)v_hdx\quad \forall v_h\in V_{h,0}^{\rm IFE}.
\end{equation}
Recalling that $a_h(\cdot,\cdot)=a(\cdot,\cdot)$ on $H^1_0(\Omega)$, and applying  (\ref{auxi}) and (\ref{pro_L21}), we have
\begin{equation}\label{pro_L22}
\begin{aligned}
\|u&-u_h\|_{L^2(\Omega)}^2=a(u,z)-A_h(u_h,z_h)=A_h(u,z)-A_h(u_h,z_h)-\widetilde{a}_h(u,z)+a_h(u,z)\\
&=A_h(u-u_h,z-z_h)+A_h(u-u_h,z_h)+A_h(u_h,z-z_h)+\left(a_h(u,z)-\widetilde{a}_h(u,z)\right),
\end{aligned}
\end{equation}
where the relation $b_h(u,z)=s_h(u,z)=0$ is used in the second identity since $[u]_e=[v]_e=0$ for all edges.
Lemma~\ref{lem_betah}  provides the estimate for the last term
\begin{equation}\label{pro_L22_1}
a_h(u,z)-\widetilde{a}_h(u,z)\leq Ch^2\|u\|_{H^2(\Omega^+\cup\Omega^-)}\|z\|_{H^2(\Omega^+\cup\Omega^-)}.
\end{equation}
The first terms on the right-hand side of (\ref{pro_L22}) can be estimated using Theorem~\ref{theo_mainH1}, 
\begin{equation}\label{pro_L22_2}
A_h(u-u_h,z-z_h)\leq \interleave u-u_h\interleave_h \interleave z-z_h\interleave_h\leq Ch^2\|u\|_{H^2(\Omega^+\cup\Omega^-)}\|z\|_{H^2(\Omega^+\cup\Omega^-)}.
\end{equation}
We rewrite the second term on the right-hand side of (\ref{pro_L22}) as
\begin{equation}\label{pro_L2_N00}
A_h(u-u_h,z_h)=A_h(u-u_h,z_h-I_h^{\rm IFE}z)+A_h(u-u_h,I_h^{\rm IFE}z_h).
\end{equation}
It is easy to see that
\begin{equation}\label{pro_L2_N0}
A_h(u-u_h,z_h-I_h^{\rm IFE}z)\leq \interleave u-u_h\interleave_h\interleave z_h-I_h^{\rm IFE}z\interleave_h\leq Ch^2\|u\|_{H^2(\Omega^+\cup\Omega^-)}\|z\|_{H^2(\Omega^+\cup\Omega^-)}.
\end{equation}
From (\ref{pro_L2_ex}), we have
\begin{equation}\label{pro_L2_N1}
A_h(u-u_h,I_h^{\rm IFE}z_h)=\sum_{e\in\mathcal{E}_h^{non}}\int_e\beta\nabla u\cdot\textbf{n}_e[I_h^{\rm IFE}z_h]_eds+\widetilde{a}_h(u,I_h^{\rm IFE}z_h)-a_h(u,I_h^{\rm IFE}z_h).
\end{equation}
Since $[z]_e=0$, the first term on the right-hand side can be estimated as 
\begin{equation}\label{pro_L2_N2}
\begin{aligned}
\sum_{e\in\mathcal{E}_h^{non}}\int_e\beta\nabla u\cdot\textbf{n}_e[I_h^{\rm IFE}z_h]_eds&=\sum_{e\in\mathcal{E}_h^{non}}\int_e(\beta\nabla u\cdot\textbf{n}_e-P_0^e(\beta\nabla u\cdot\textbf{n}_e))[I_h^{\rm IFE}z_h-z]_eds\\
&\leq Ch^2\|u\|_{H^2(\Omega^+\cup\Omega^-)}\|z\|_{H^2(\Omega^+\cup\Omega^-)},
\end{aligned}
\end{equation}
where the Cauchy-Schwarz inequality, (\ref{H1_4}), the standard trace inequality and Theorem~\ref{theorem_interpolation} are used. Applying Lemma~\ref{lem_betah} and Theorem~\ref{theorem_interpolation} again we obtain 
\begin{equation}\label{pro_L2_N3}
\begin{aligned}
|\widetilde{a}_h(u,I_h^{\rm IFE}z_h)-a_h(u,I_h^{\rm IFE}z_h)|&\leq |\widetilde{a}_h(u,z)-a_h(u,z)|\\
&\quad+|\widetilde{a}_h(u,I_h^{\rm IFE}z_h-z)-a_h(u,I_h^{\rm IFE}z_h-z)|\\
&\leq Ch^2\|u\|_{H^2(\Omega^+\cup\Omega^-)}\|z\|_{H^2(\Omega^+\cup\Omega^-)}.
\end{aligned}
\end{equation}
Combining (\ref{pro_L2_N00})-(\ref{pro_L2_N3}), we find
\begin{equation}\label{pro_L22_3}
A_h(u-u_h,z_h)\leq Ch^2\|u\|_{H^2(\Omega^+\cup\Omega^-)}\|z\|_{H^2(\Omega^+\cup\Omega^-)},
\end{equation}
and similarly,
\begin{equation}\label{pro_L22_4}
A_h(u_h,z-z_h)\leq Ch^2\|u\|_{H^2(\Omega^+\cup\Omega^-)}\|z\|_{H^2(\Omega^+\cup\Omega^-)}.
\end{equation}
Applying  (\ref{pro_L22})-(\ref{pro_L22_2}), (\ref{pro_L22_3})-(\ref{pro_L22_4}), we arrive at the estimate
\begin{equation*}
\|u-u_h\|^2_{L^2(\Omega)}\leq Ch^2\|u\|_{H^2(\Omega^+\cup\Omega^-)}\|z\|_{H^2(\Omega^+\cup\Omega^-)},
\end{equation*}
which together with the regularity result (\ref{reg_L2}) implies the estimate (\ref{L2_err_method2}).
\end{proof}

\section{Numerical examples}\label{sec_num}
In this section, we present some numerical examples to validate the theoretical analysis.  To avoid redundancy, we only report numerical results of IFE methods based on the Crouzeix-Raviart  element since the  results of  IFE methods based on  the rotated-$Q_1$ element are almost the same. 
We examine the convergence rate of IFE solutions using the following norms
\begin{equation*}
|e_h|_{H^1}:=\left(\sum_{T\in\mathcal{T}_h}\|\sqrt{\beta_h}\nabla (u-u_h)\|^2_{L^2(T)}\right)^{1/2} \qquad \mbox{ and }\qquad \|e_h\|_{L^2}:=\|u-u_h\|_{L^2(\Omega)}.
\end{equation*}

For comparison,  we replace $\beta(x)$ by $\beta_h(x)$ in the nonconforming IFE method (IFEM) without penalties (\ref{method1}) in our computation. Thus, the difference between  the  nonconforming IFEM without penalties and our new nonconforming IFEM  (\ref{method2})  is the terms $b_h(\cdot,\cdot)$ and $s_h(\cdot,\cdot)$.  In view of the analysis for our new method,  the error resulting from  replacing $\beta(x)$ by $\beta_h(x)$ does not affect the error estimates  in Theorem~\ref{the_H1_err} for the  nonconforming IFEM without penalties.

In all numerical examples, we set $\Omega=(-1,1)\times(-1,1)$ and use uniform meshes obtained as follows. We first partition the domain into $N\times N$ congruent rectangles, and then obtain the triangulation by cutting the rectangles along one of diagonals in the same direction (see Figure~\ref{interface_edge}).  The interface $\Gamma$ and the subdomains $\Omega^\pm$ are determined by a given function $\varphi(x_1,x_2)$, i.e., $\Gamma=\{(x_1,x_2)\in \Omega : \varphi(x_1,x_2)=0\}$, $\Omega^+=\{(x_1,x_2)\in \Omega : \varphi(x_1,x_2)>0\}$ and $\Omega^-=\{(x_1,x_2)\in \Omega : \varphi(x_1,x_2)<0\}$.

\subsection{A counter example with $\nabla u\cdot {\rm\mathbf{t}}\not=0$ on $\Gamma$}\label{ex1}
We use this example to show that the nonconforming IFEM without penalties  does not converge optimally, although the integral values on edges are used as degrees of freedom. 

 {\em Example}~\ref{ex1}.  We set $\varphi(x_1,x_2)= x_1^2+x_2^2-r_0^2$. Let $(r,\theta)$ be the polar coordinate of  $x=(x_1,x_2)$. The exact solution is chosen as $u(x)=j(x)v(x)\omega(x)$, where $\omega(x)=\sin(\theta)$,
\begin{equation*}
j(x)=\left\{
\begin{aligned}
&\exp\left(-\frac{1}{1-(r-r_0)^2/\eta^2}\right)~&\mbox{ if } |r-r_0|<\eta,\\
&0 &\mbox{ if } |r-r_0|\geq\eta,
\end{aligned}\right.
\end{equation*}
and
\begin{equation*}
v(x)=\left\{
\begin{aligned}
&1+(r^2-r_0^2)/\beta^+(x)~&\mbox{ if } x\in\Omega^+,\\
&1+(r^2-r_0^2)/\beta^-(x)~&\mbox{ if } x\in\Omega^-.
\end{aligned}\right.
\end{equation*}
Let $r_0=0.5$, $\eta=0.45$, $\beta^+(x)$ and $\beta^-(x)$ be positive constants. It is easy to verify that the jump condition (\ref{p1.2})-(\ref{p1.3}) is satisfied and $\nabla u\cdot {\rm\mathbf{t}}\not=0$ on $\Gamma$. We test two cases:  $(\beta^+,\beta^-)=(10,1000)$ and  $(\beta^+,\beta^-)=(1000,10)$. The exact solutions of these two cases are plotted in Figure~\ref{ex1_fig}. 

\begin{figure} [htbp]
\centering
\subfigure{ 
\includegraphics[width=0.45\textwidth]{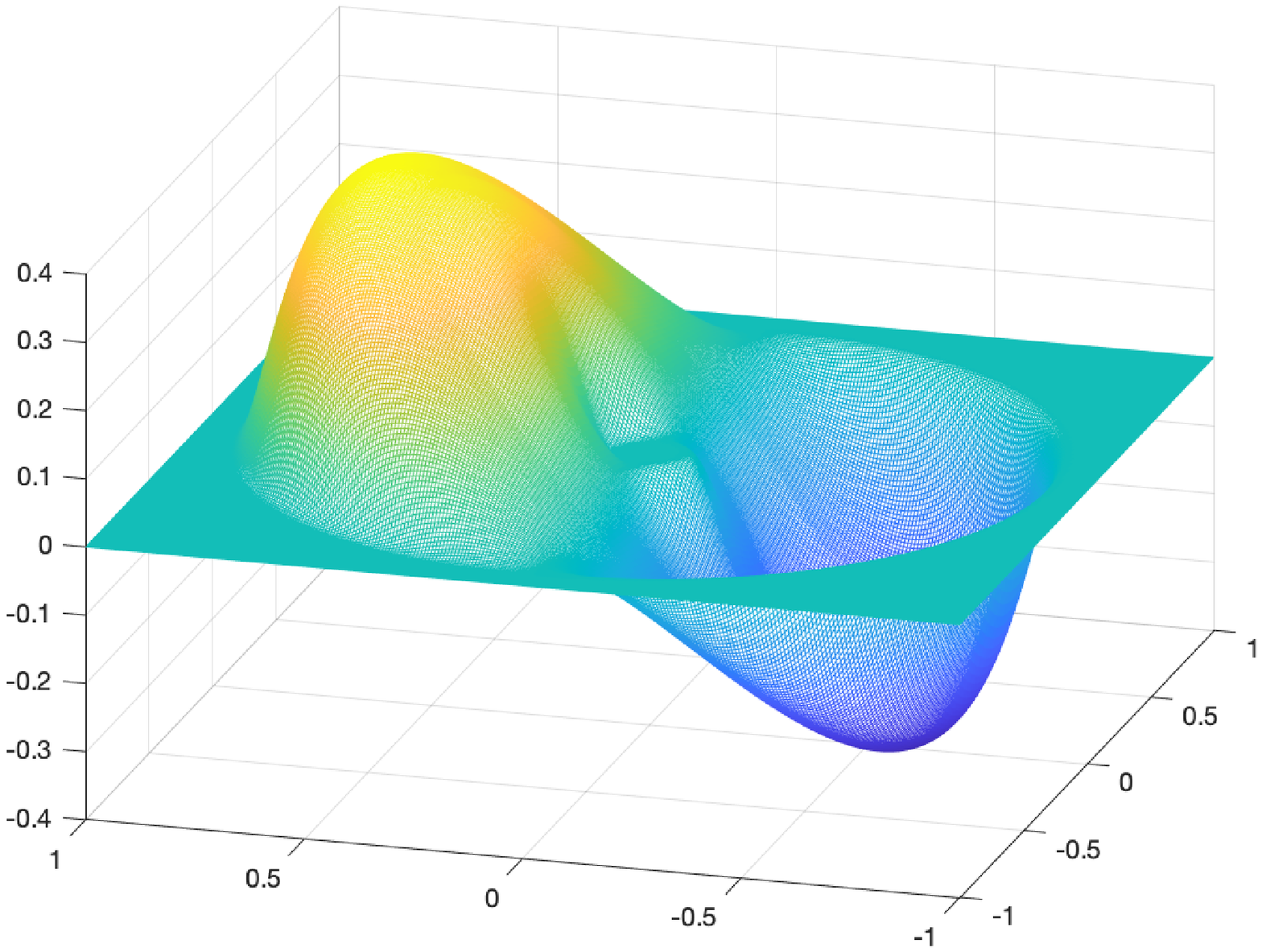}}
\subfigure{
\includegraphics[width=0.45\textwidth]{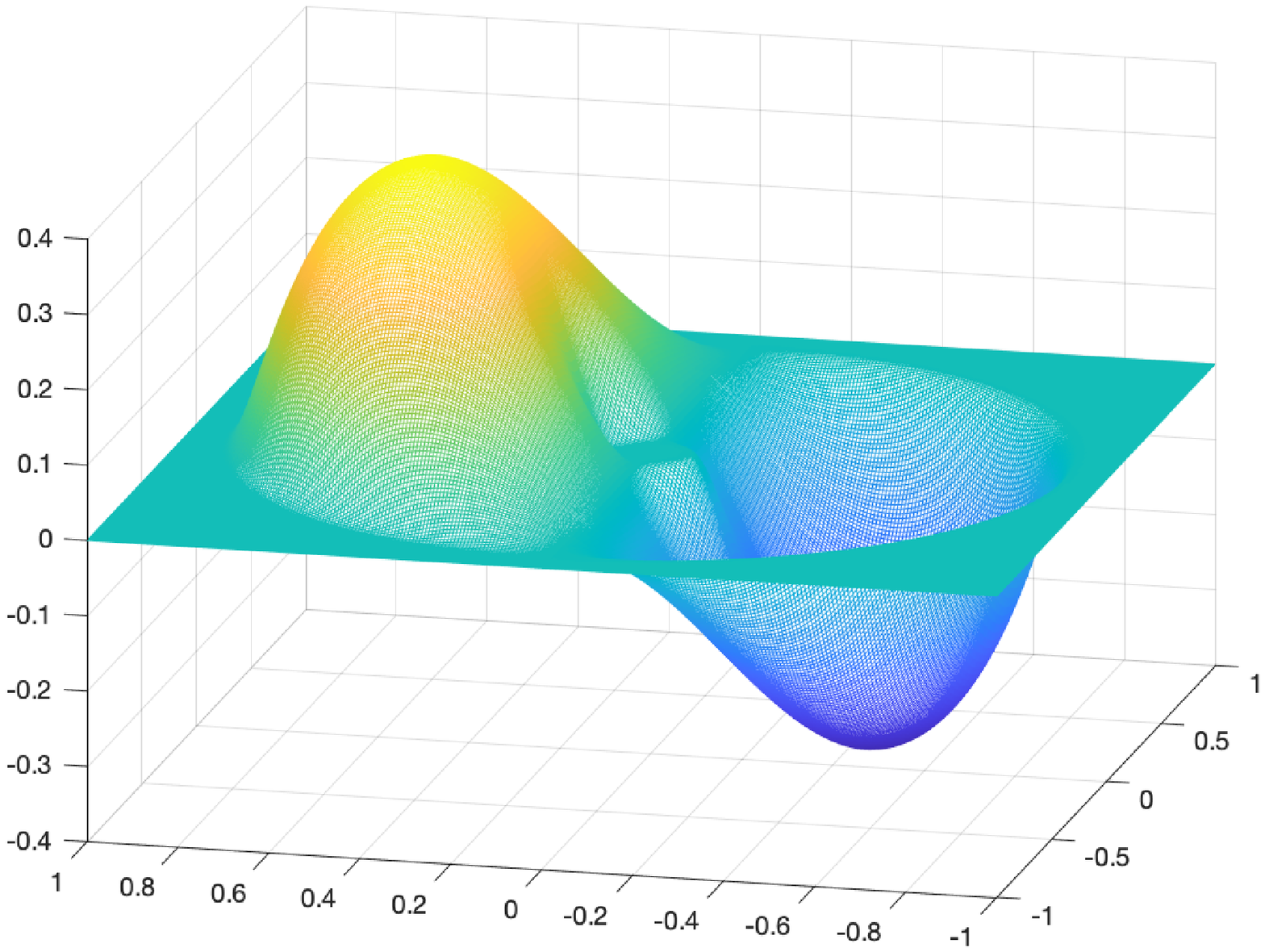}}
 \caption{Exact solutions of  {\em Example}~\ref{ex1}. Left: $(\beta^+,\beta^-)=(10,1000)$; Right: $(\beta^+,\beta^-)=(1000,10).$ \label{ex1_fig}} 
\end{figure}

We report numerical results in Tables~\ref{ex1_biao11000}-\ref{ex1_biao10001} which clearly confirm our theoretical analysis. The second and third columns in Tables~\ref{ex1_biao11000}-\ref{ex1_biao10001} indicate suboptimal convergence rates: $\|e_h\|_{L^2}\approx O(h)$,  $|e_h|_{H^1}\approx O(h^{1/2})$ for the nonconforming IFEM without penalties (\ref{method1}). When the terms $b_h(\cdot,\cdot)$ and $s_h(\cdot,\cdot)$ are added to the scheme, i.e., the new nonconforming IFEM (\ref{method2}), we observe the optimal convergence rates (see last two columns in Tables~\ref{ex1_biao11000}-\ref{ex1_biao10001}).

We also use this example to test the classic IFE method \cite{Li2003new, he2012convergence} where the nodal values are used as degrees of freedom  and no penalties are included. The suboptimal convergence rates are also observed (i.e., $O(h^{1/2})$ in the $H^1$ norm and $O(h)$ in the $L^2$ norm), which indicates that the error estimate in \cite{he2012convergence} is sharp.  To avoid redundancy, we do not list the numerical results here.

\begin{table}[H]
\caption{Numerical results of  {\em Example}~\ref{ex1} with $(\beta^+,\beta^-)=(10,1000)$.\label{ex1_biao11000}}
\begin{center}
\begin{tabular}{|c|c c|c c|c c|c c|}
  \hline
  & \multicolumn{4}{c|}{ Nonconforming IFEM without penalties} & \multicolumn{4}{c|}{New nonconforming IFEM}\\ \hline
       $N$  &  $\|e_h\|_{L^2}$   &  rate   &  $|e_h|_{H^1}$  &  rate  &  $\|e_h\|_{L^2}$  &  rate  & $|e_h|_{H^1}$  &  rate    \\ \hline
      8  & 2.221E$-$01  &         & 2.140E$+$01   &            & 1.617E$-$01    &               &  1.781E$+$01   &                   \\ \hline
    16  & 7.650E$-$02  & 1.54  & 1.037E$+$01  & 1.05    &  4.414E$-$02   &    1.87   &   6.889E$+$00   &   1.37    \\ \hline
    32  & 1.745E$-$02  & 2.13  & 5.970E$+$00  & 0.80    &  5.989E$-$03   &   2.88   &   3.851E$+$00   &   0.84    \\ \hline
    64  & 7.322E$-$03  & 1.25  & 3.597E$+$00  & 0.73    &  7.855E$-$04   &   2.93   &   1.784E$+$00   &   1.11    \\ \hline
  128  & 3.204E$-$03  & 1.19  & 2.309E$+$00  & 0.64    &  1.935E$-$04   &   2.02   &   8.932E$-$01    &   1.00    \\ \hline
  256  & 1.514E$-$03  & 1.08  & 1.548E$+$00  & 0.58    &  4.836E$-$05   &   2.00   &   4.461E$-$01    &   1.00    \\ \hline
  512  & 7.276E$-$04  & 1.06  & 1.056E$+$00  & 0.55    &  1.197E$-$05   &   2.01   &   2.229E$-$01    &   1.00    \\ \hline
1024  & 3.603E$-$04  & 1.01  &  7.378E$-$01  & 0.52    &  2.992E$-$06   &   2.00  &    1.114E$-$01    &   1.00    \\ \hline
\end{tabular}
\end{center}
\end{table}

\begin{table}[H]
\caption{Numerical results of  {\em Example}~\ref{ex1} with $(\beta^+,\beta^-)=(1000,10)$.\label{ex1_biao10001}}
\begin{center}
\begin{tabular}{|c|c c|c c|c c|c c|}
  \hline
  & \multicolumn{4}{c|}{Nonconforming IFEM without penalties} & \multicolumn{4}{c|}{New nonconforming IFEM}\\ \hline
       $N$  &  $\|e_h\|_{L^2}$   &  rate   &  $|e_h|_{H^1}$  &  rate  &  $\|e_h\|_{L^2}$  &  rate  & $|e_h|_{H^1}$  &  rate    \\ \hline
    8   &  1.736E$-$01   &          & 4.641E$+$01   &          &  1.540E$-$01  &             &    4.611E$+$01   &                      \\ \hline
   16  &  7.679E$-$02   & 1.18  &   1.686E$+$01 & 1.46  &  6.402E$-$02  &    1.27   &   1.579E$+$01  &    1.55     \\ \hline
   32  &  1.186E$-$02   & 2.69  &  8.927E$+$00  &  0.92 &  5.868E$-$03  &    3.45   &   8.055E$+$00   &   0.97     \\ \hline
   64  &  5.188E$-$03   & 1.19  &  4.822E$+$00  &  0.89 &  8.575E$-$04  &    2.77   &   3.888E$+$00    &  1.05     \\ \hline
  128 &  2.242E$-$03   &  1.21 &  2.802E$+$00  &  0.78 &  1.944E$-$04  &    2.14   &    1.949E$+$00   &   1.00     \\ \hline
  256 &  1.061E$-$03   &  1.08 &  1.746E$+$00  &  0.68 &  4.841E$-$05  &    2.01   &    9.738E$-$01    &  1.00     \\ \hline
  512 &  5.043E$-$04   &  1.07 &  1.133E$+$00  &  0.62 &  1.218E$-$05  &    1.99   &    4.868E$-$01    &  1.00     \\ \hline
 1024&  2.483E$-$04   &  1.02 &  7.654E$-$01   &  0.57 &  2.982E$-$06  &    2.03   &    2.434E$-$01    &  1.00     \\ \hline
\end{tabular}
\end{center}
\end{table}

\subsection{An example with variable coefficients and a non-convex interface}\label{ex2}
 {\em Example}~\ref{ex2}. We set
 $$\varphi(x_1,x_2)=(3(x_1^2+x_2^2)-x_1)^2-x_1^2-x_2^2+0.02.$$
 The exact solution is chosen as $u(x)=\varphi(x)/\beta(x)$, where
 \begin{equation*}
 \beta(x_1,x_2)=\left\{
 \begin{aligned}
 &\beta^+(x_1,x_2)=300(2+\sin(6x_1+6x_2))& \mbox{ if } \varphi(x_1,x_2)>0,\\
 &\beta^-(x_1,x_2)=2+\cos(6x_1+6x_2) &\mbox{ if } \varphi(x_1,x_2)<0.
 \end{aligned}\right.
 \end{equation*}
 It is easy to verify that the jump condition (\ref{p1.2})-(\ref{p1.3}) is satisfied and $\nabla u\cdot {\rm\mathbf{t}}=0$ on $\Gamma$. The exact solution and the interface are plotted in Figure~\ref{ex2_fig}.
 
  \begin{figure} [htbp]
\centering
\subfigure{ 
\includegraphics[width=0.45\textwidth]{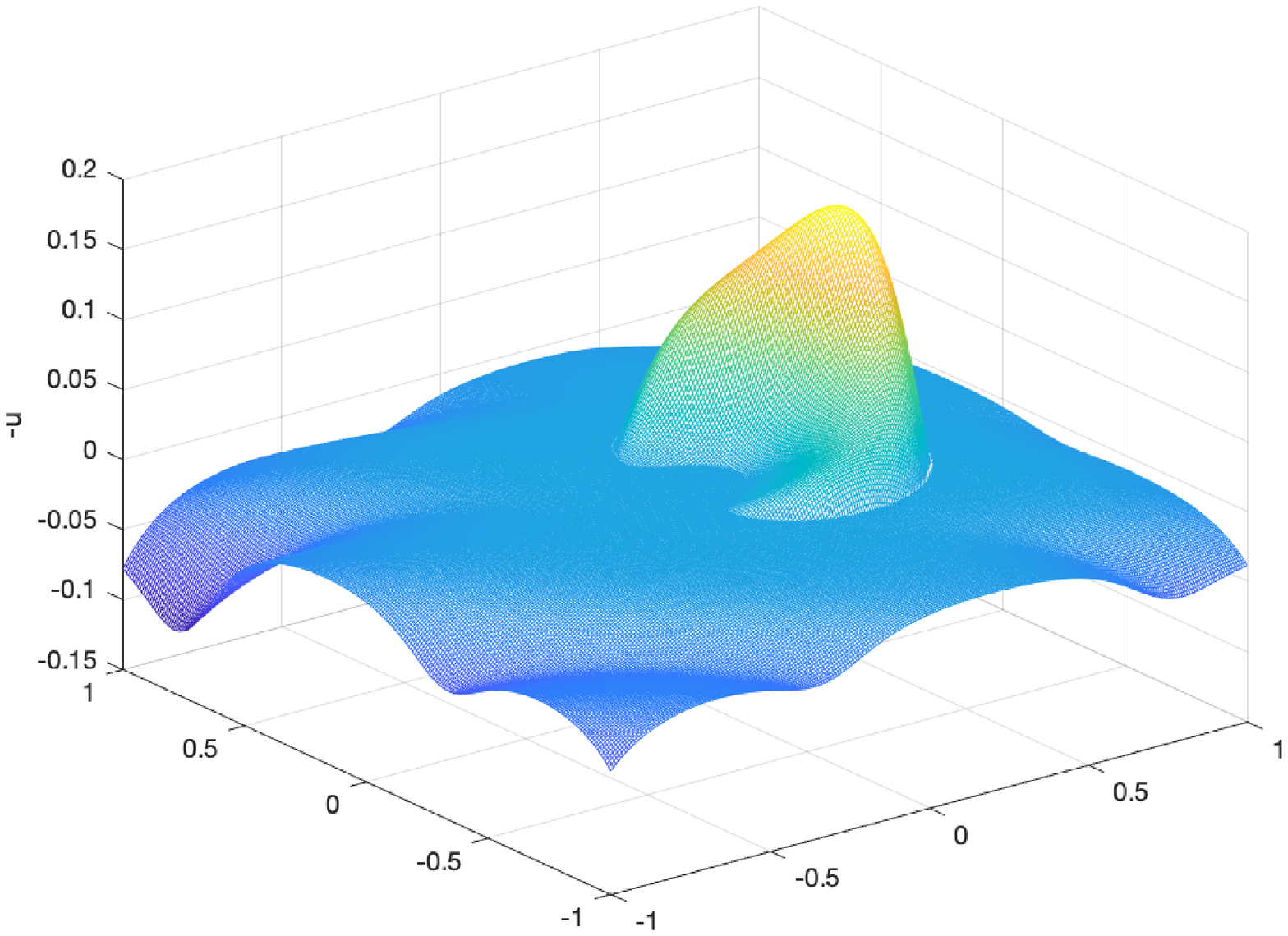}}
\subfigure{
\includegraphics[width=0.45\textwidth]{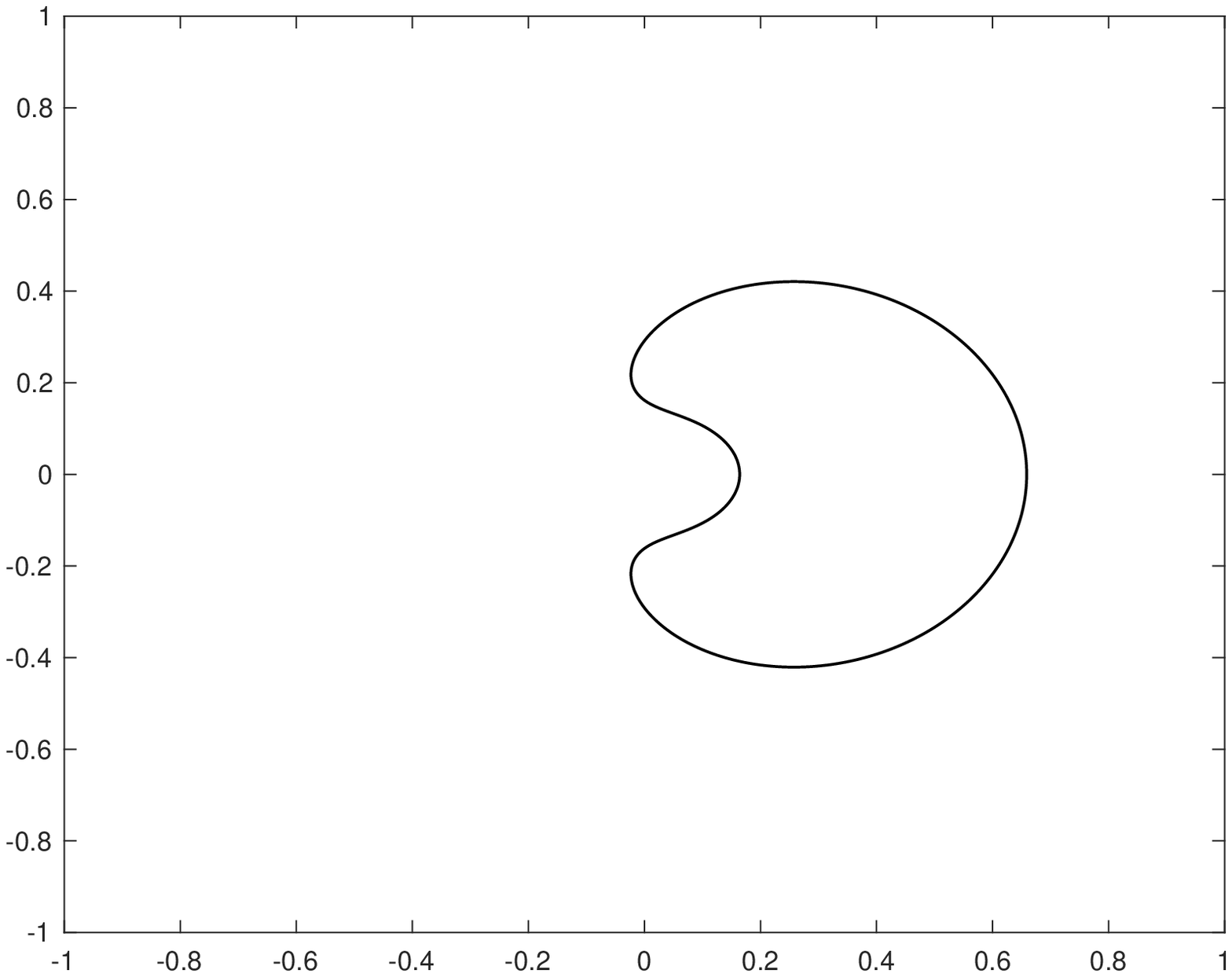}}
 \caption{ $-u$ (left) and the interface $\Gamma$ (right) for {\em Example}~\ref{ex2}. \label{ex2_fig}} 
\end{figure}

To deal with variable coefficients, we choose $\beta^+_c=\beta^+(x_m)$, $\beta^-_c=\beta^-(x_m)$ on each interface element $T\in\mathcal{T}_h^\Gamma$, where $x_m$ is the midpoint of $\Gamma_h\cap T$. Since $\nabla u\cdot {\rm\mathbf{t}}=0$ on $\Gamma$, our theoretical analysis suggests the optimal convergence rates for both the nonconforming IFEM without penalties and  our new nonconforming IFEM, which are confirmed by the results shown in Figure~\ref{ex2_fig_er}.  We also test the
parameter-free partially penalized immersed method (PPIFEM) using nodal values \cite{2021ji_IFE} and
the nonconforming PPIFEM in \cite{zhangphd} with $\eta_e=10\mathrm{max}(\beta^+(x_\Gamma), \beta^-(x_\Gamma))$, $x_\Gamma=e\cap\Gamma$ for each interface edge $e$ (see Remark~\ref{remark_PPIFEM}). The numerical results in  Figure~\ref{ex2_fig_er} show that the convergence orders of all IFEMs are optimal. 

 \begin{figure} [htbp]
\centering
\includegraphics[width=0.8\textwidth]{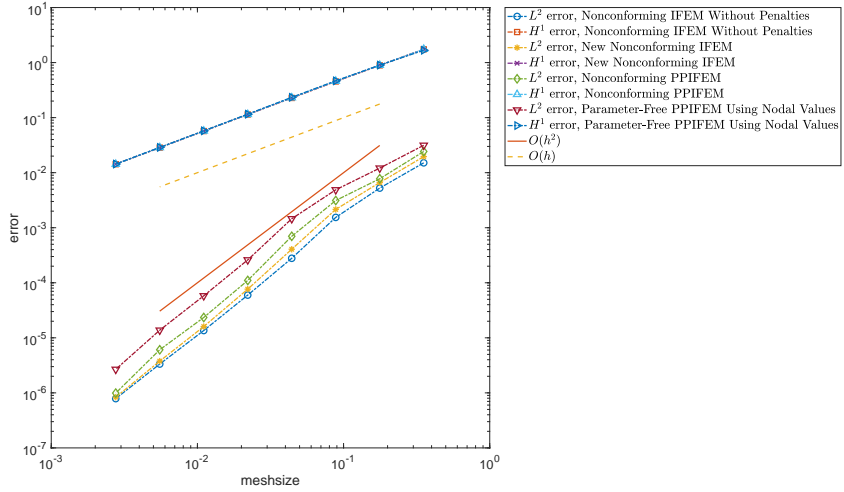}
 \caption{The convergence history of different IFEMs for {\em Example}~\ref{ex2}.\label{ex2_fig_er}} 
\end{figure}

\subsection{An example with a straight  interface and a piecewise linear solution} \label{ex3}

Inspired by Remark~\ref{remark_suboptimal}, we construct the following example with a straight  interface and a piecewise linear solution.

{\em Example}~\ref{ex3}. 
We choose 
\begin{equation*}
\varphi(x_1,x_2)=\frac{x_1-x_2}{\sqrt{2}},\quad u^\pm(x_1,x_2)=\frac{x_1+x_2}{\sqrt{2}}+\frac{\varphi(x_1,x_2)}{\beta^\pm}.
\end{equation*}
We set $\beta^+=2$ and $\beta^-=1$. Note that the interface cut the boundary of the computational domain. Obviously, $\nabla u\cdot {\rm\mathbf{t}}\not=0$ on $\Gamma$. Numerical results reported in Table~\ref{biao_ex3} clearly show the suboptimal convergence of the nonconforming IFEM without penalties. However, the errors obtained by our new nonconforming IFEM are near machine precision.

\begin{table}[H]
\caption{Numerical results of  {\em Example}~\ref{ex3}.\label{biao_ex3}}
\begin{center}
\begin{tabular}{|c|c c|c c|c |c|}
  \hline
  & \multicolumn{4}{c|}{ Nonconforming IFEM without penalties} & \multicolumn{2}{c|}{New nonconforming IFEM}\\ \hline
          8    &   1.249E$-$02    &              &  1.996E$-$01  &               &   7.971E$-$16 &     3.172E$-$15       \\ \hline   
        16    & 4.561E$-$03      & 1.45      & 1.460E$-$01   &   0.45     &  1.207E$-$15  &     5.528E$-$15      \\ \hline
        32    &  1.638E$-$03     &  1.48      & 1.049E$-$01  &   0.48     &  2.084E$-$15  &    1.072E$-$14      \\ \hline 
\end{tabular}
\end{center}
\end{table}

\subsection{An example with nonhomogeneous jump conditions and $\nabla u\cdot {\rm\mathbf{t}}\not=0$ on $\Gamma$} \label{ex4}
Almost all the works in the literature only consider the examples satisfying the condition $\nabla u\cdot \mathbf{t} = 0$ on $\Gamma$. The reason is that it is not easy to construct a function  satisfying  $\nabla u\cdot \mathbf{t} \not= 0$ and the homogeneous conditions (\ref{p1.2})-(\ref{p1.3}) on a curved interface simultaneously. For the problem with nonhomogeneous jump conditions, the relation $\nabla u\cdot \mathbf{t} = 0$ can be easily violated.  In the final example, we consider this case.

{\em Example}~\ref{ex4}.
We set 
\begin{equation*}
\begin{aligned}
&\varphi(x_1,x_2)=x_1^2+x_2^2-0.5^2,\\
&\beta^+(x_1,x_2)=\sin(x_1+x_2)+2,\quad \beta^-(x_1,x_2)=\cos(x_1+x_2)+2,\\
&u^+(x_1,x_2)=\ln(x_1^2+x_2^2),~\quad\quad\quad u^-(x_1,x_2)=\sin(x_1+x_2).
\end{aligned}
\end{equation*}
Clearly, $g_D:=[u]_{\Gamma}\not=0$ and $g_N:=[\beta\nabla u\cdot \textbf{n}]_{\Gamma}\not=0$. To deal with the nonhomogeneous jump conditions, we need a correction function $u_h^J$ defined as follows. We define $u_h^J|_T=0$ if $T\in\mathcal{T}_h^{non}$. On an interface element $T$, similarly to (\ref{shape2.1}) and (\ref{shape2}), we define $u_h^J|_{T_h^\pm}=\phi^\pm$ with $\phi^\pm\in V_h(T)$ satisfying 
\begin{align}
&\phi^+(x)-\phi^-(x)=g_D(x)\quad \forall x \in \{D,E \},\\
&\beta_c^+ (\nabla \phi^+ \cdot \textbf{n}_h)(x_p)-\beta_c^- (\nabla \phi^- \cdot \textbf{n}_h)(x_p)=\frac{g_N(D)+g_N(E)}{2},\label{jp2_nonhom}\\
&N_i(u_h^J|_T)=0\quad \forall i\in\mathcal{I},
\end{align}
where $D$ and $E$ are endpoints of $\Gamma_h\cap T$. Now, for both methods (\ref{method1}) and  (\ref{method2}), we replace the discrete trial space $V_{h,0}^{\rm IFE}$  by $V_{h,g}^{\rm IFE}+\{u_h^J\}$, where $V_{h,g}^{\rm IFE}$ is the nonconforming IFE space satisfying  the nonhomogeneous Dirichlet boundary condition $g:=u|_{\partial \Omega}$ on boundary edges approximately. The numerical results reported in Table~\ref{ex4_biao} show that the new nonconforming IFEM converges optimally in both $L^2$ and $H^1$ norms; however, the nonconforming IFEM without penalties only has the suboptimal convergence since $\nabla u\cdot \mathbf{t} \not= 0$ on $\Gamma$.

We note that in (\ref{jp2_nonhom}) we use point values $g_N(D)$ and $g_N(E)$, which are not well defined if $g_N\in H^{1/2}(\Gamma)$. In this numerical example, $g_N(D)$ and $g_N(E)$ are meaningful since $g_N\in C(\Gamma)$. For the general case with $g_N\in H^{1/2}(\Gamma)$ and $g_D\in H^{3/2}(\Gamma)$, the point values should be replaced by averages in some sense. The corresponding IFE methods and  theoretical analysis will be presented in a forthcoming paper.

\begin{table}[H]
\caption{Numerical results of  {\em Example}~\ref{ex4}.\label{ex4_biao}}
\begin{center}
\begin{tabular}{|c|c c|c c|c c|c c|}
  \hline
  & \multicolumn{4}{c|}{Nonconforming IFEM without penalties} & \multicolumn{4}{c|}{New nonconforming IFEM}\\ \hline
       $N$  &  $\|e_h\|_{L^2}$   &  rate   &  $|e_h|_{H^1}$  &  rate  &  $\|e_h\|_{L^2}$  &  rate  & $|e_h|_{H^1}$  &  rate    \\ \hline
     8    &   7.341E$-$02   &              &  1.194E$+$00    &              &    6.660E$-$02   &              &  1.052E$+$00   &  \\ \hline
   16    &   2.353E$-$02   &    1.64   &    7.886E$-$01  &     0.60  &     1.659E$-$02  &     2.01  &     5.385E$-$01 &      0.97 \\ \hline
   32   &    8.092E$-$03   &    1.54   &    5.171E$-$01   &    0.61  &   4.293E$-$03   &    1.95    &   2.738E$-$01   &    0.98 \\ \hline
   64   &    3.282E$-$03    &   1.30    &   4.082E$-$01   &    0.34  &  1.038E$-$03    &   2.05    &   1.381E$-$01    &   0.99  \\ \hline
   128   &    1.158E$-$03    &   1.50   &    2.691E$-$01   &    0.60  & 2.699E$-$04    &   1.94    &   6.945E$-$02    &   0.99    \\ \hline
   256   &    4.423E$-$04    &   1.39   &    1.906E$-$01   &    0.50 &  6.714E$-$05    &   2.01    &   3.483E$-$02    &   1.00   \\ \hline
   512   &    1.780E$-$04    &   1.31   &    1.365E$-$01   &    0.48 &  1.685E$-$05    &   1.99    &   1.744E$-$02    &   1.00   \\ \hline
  1024  &     7.517E$-$05    &   1.24  &     9.795E$-$02   &    0.48 &    4.198E$-$06  &     2.00   &    8.727E$-$03  &     1.00   \\ \hline
\end{tabular}
\end{center}
\end{table}

\section{Conclusions}
In this paper, we have shown  that the nonconforming IFE methods using the integral-value degrees of freedom on edges are not guaranteed to  achieve optimal convergence rates without adding penalties although the continuity of IFE shape functions is weakly enforced through average values over edges.  The suboptimal convergence rates have been confirmed by a counter numerical example where the the tangential derivative of the exact solution is not zero on the interface.  We think there is a similar issue for nonconforming IFE methods using integral-values on edges as degrees of freedom for solving elasticity and Stokes interface problems, which is an interesting topic in our future research.

To recover the optimal convergence rates, we have developed a new nonconforming IFE method with additional terms at interface edges. The new nonconforming IFE method is symmetric and the coercivity is ensured by a local lifting operator without a sufficiently large penalty parameter. We have also proved that IFE basis functions based on the Crouzeix-Raviart elements are unisolvent on arbitrary triangles which is one of advantages compared with the IFEs using nodal values as degrees of freedom. The optimal approximation capabilities of nonconforming IFE spaces based on the Crouzeix-Raviart and the rotated-$Q_1$  elements have been derived via a novel approach which can handle the case of variable coefficients easily.  The optimal error estimates for the IFE solutions in the $H^1$- and $L^2$- norms have been derived and confirmed by some numerical examples. 

\textbf{Acknowledgments}:
The work of the first author was supported by  the National Natural Science Foundation of China grants 11701291, 12101327, and the Natural Science Foundation of Jiangsu Province grant BK20200848. The work of the second and third authors was supported by the National Natural Science Foundation of China grants 12071227, 11871281, 11731007, and the Natural Science Foundation of the Jiangsu Higher Education Institutions of China grant 20KJA110001. The work of the fourth author was supported by a Simons grant  633724.


\appendix
\section{Proof of optimal approximation capabilities of IFE spaces}
Following the notations in the paper, 
on an interface element $T\in\mathcal{T}_h^\Gamma$, we define IFE basis functions  as
\begin{equation}\label{basis}
\phi_i\in S_h(T),\quad  N_j(\phi_i)=\delta_{ij}~~~~ \forall i,j\in\mathcal{I}.
\end{equation}
Also define functions $\phi_i^s\in V_h(T)$, $s=+,-$, $i\in\mathcal{I}$ such that $\phi_i^s=\phi_i|_{T_h^s}$. 
Let $\lambda_i$ be traditional  basis functions, i.e.,
\begin{equation}\label{basis_standard}
\lambda_i\in V_h(T),\quad  N_j(\lambda_i)=\delta_{ij}~~~~\forall  i,j\in\mathcal{I}.
\end{equation}
 Note that these functions  depend on  elements. We  omit this dependence in our notation for simplicity. 
It is well-known that the traditional  basis functions satisfy 
\begin{equation}\label{basis_stad_est}
|\lambda_i|_{W^m_\infty(T)}\leq Ch_T^{-m},\quad i\in \mathcal{I}, ~m=0,1,2,
\end{equation}
where the constant $C$ only depends on the shape regularity parameter $\varrho$.
The following lemma shows that the IFE basis functions also have similar estimates, which is one of essential ingredients for the success of IFE methods.
\begin{lemma} 
There exists a constant $C$, depending only on $\beta_c^+$, $\beta_c^-$ and the shape regularity parameter $\varrho$, such that
\begin{align}
&|\phi_i^+|_{W^m_\infty(T)}\leq Ch_T^{-m},~~|\phi_i^-|_{W^m_\infty(T)}\leq Ch_T^{-m}, \quad i\in \mathcal{I}, ~m=0,1,2,\label{bounds_basis1}\\
&|\phi_i|_{W^m_\infty(T_h^+\cup T_h^-)}\leq Ch_T^{-m},\quad i\in \mathcal{I}, ~m=0,1,2.\label{bounds_basis2}
\end{align}
\end{lemma}
\begin{proof}
See Theorem 4.2 in \cite{Guo2018CMA} for the estimate (\ref{bounds_basis2}) for the rotated-$Q_1$ element and the Crouzeix-Raviart element on right triangles.  The estimate (\ref{bounds_basis1}) can also be obtained easily from the proof of Theorem 5.6 in \cite{GuoIMA2019}. Next, we give a proof for general triangles without the constraint $\alpha_{max}=\pi/2$. We just need to  prove (\ref{bounds_basis1}) since (\ref{bounds_basis2})  is a direct consequence of (\ref{bounds_basis1}).  Using $|N_j(\phi_i)|\leq 1$, $\|\nabla \lambda_i\|_{L^\infty(T)}\leq Ch_T^{-1}$, $i,j\in\mathcal{I}$,  we can estimate vectors in (\ref{shap_eq}) as
\begin{equation*}
\|\boldsymbol{b}\|\leq C,~ \|\boldsymbol{\gamma}\|\leq Ch_T^{-1},~\|\boldsymbol{\delta}\|\leq Ch_T,
\end{equation*}
which, together with (\ref{est_max}) and (\ref{soltu}) lead to $\|\boldsymbol{c}\|\leq C$, where the constant $C$ is independent of $h_T$ and the interface location relative to the mesh. Thus, from (\ref{pro_phi}), it follows 
\begin{equation*}
|\phi_i^+|_{W^m_\infty(T)}\leq Ch_T^{-m},\quad i\in \mathcal{I}, ~m=0,1,
\end{equation*}
where we have used the fact $ |\lambda_i|_{W^m_\infty(T)}\leq Ch_T^{-m}$. It follows from (\ref{exc_0}) that $|c_0|\leq Ch_T^{-1}$ which together with  (\ref{pro_phi}) yields
\begin{equation*}
|\phi_i^-|_{W^m_\infty(T)}\leq Ch_T^{-m},\quad i\in \mathcal{I}, ~m=0,1.
\end{equation*}
\end{proof}

Given two functions $v^+\in L^2(T)$ and $v^-\in L^2(T)$, we define
\begin{equation*}
[\![v^\pm]\!](x):=v^+(x)-v^-(x)\qquad \forall x\in T.
\end{equation*}
%
We next introduce {\em auxiliary functions}  on each interface element $T\in\mathcal{T}_h^\Gamma$.
Recalling that $D$ and $E$ are intersection points of $\Gamma$ and $\partial T$, define auxiliary functions $\Upsilon_1(x)$, $\Upsilon_2(x)$ and $\Psi(x)$ as
\begin{equation}\label{def_Upsilon1}
\Upsilon_i(x):=\left\{
\begin{aligned}
\Upsilon^+_i(x)\in V_h(T)\quad \mbox{ if } x=(x_1,x_2)\in T_h^+,\\
\Upsilon^-_i(x)\in V_h(T)\quad \mbox{ if } x=(x_1,x_2)\in T_h^-,
\end{aligned}\qquad i=1,2,
\right.
\end{equation}
such that
\begin{equation}\label{def_Upsilon2}
\begin{aligned}
&N_j(\Upsilon_i)=0~\forall j\in\mathcal{I}, \quad [\![\Upsilon_i^\pm]\!](D)=0,\quad i=1,2,\\
& [\![\beta_c^\pm \nabla\Upsilon_1^\pm \cdot \textbf{n}_h]\!](x_p)=1,\quad [\![\nabla\Upsilon_1^\pm \cdot \textbf{t}_h]\!](x_p)=0, \quad [\![d(\Upsilon_1^\pm)]\!]=0,\\
& [\![\beta_c^\pm \nabla\Upsilon_2^\pm \cdot \textbf{n}_h]\!](x_p)=0,\quad [\![\nabla\Upsilon_2^\pm \cdot \textbf{t}_h]\!](x_p)=1,\quad [\![d(\Upsilon_2^\pm)]\!]=0,
\end{aligned}
\end{equation}
and
\begin{equation}\label{def_Psi1}
\Psi(x):=\left\{
\begin{aligned}
\Psi^+(x)\in V_h(T)\quad  \mbox{ if } x=(x_1,x_2)\in T_h^+,\\
\Psi^-(x)\in V_h(T)\quad  \mbox{ if } x=(x_1,x_2)\in T_h^-,
\end{aligned}
\right.
\end{equation}
such that
\begin{equation}\label{def_Psi2}
\begin{aligned}
&N_j (\Psi)=0~\forall j\in\mathcal{I}, \quad [\![\Psi^\pm]\!](D)=1, \\
&[\![\beta_c^\pm \nabla\Psi^\pm \cdot \textbf{n}_h]\!](x_p)=0,\quad [\![\nabla\Psi^\pm \cdot \textbf{t}_h]\!](x_p)=0,\quad [\![d(\Psi^\pm)]\!]=0,\\
\end{aligned}
\end{equation}
where $d(\cdot)$ is defined in (\ref{def_d}) and  the point $x_p\in\Gamma_h\cap T$ is the same as that in ($\ref{shape2}$). 
For the rotated-$Q_1$ element, we need another  auxiliary function 
\begin{equation}\label{def_Theta1}
\Theta(x):=\left\{
\begin{aligned}
\Theta^+(x)\in V_h(T)\quad  \mbox{ if } x=(x_1,x_2)\in T_h^+,\\
\Theta^-(x)\in V_h(T)\quad  \mbox{ if } x=(x_1,x_2)\in T_h^-,
\end{aligned}
\right.
\end{equation}
such that
\begin{equation}\label{def_Theta2}
\begin{aligned}
&N_j (\Theta)=0~\forall j\in\mathcal{I}, \quad [\![\Theta^\pm]\!](D)=0, \\
&[\![\beta_c^\pm \nabla\Theta^\pm \cdot \textbf{n}_h]\!](x_p)=0,\quad [\![\nabla\Theta^\pm \cdot \textbf{t}_h]\!](x_p)=0,\quad [\![d(\Theta^\pm)]\!]=1.
\end{aligned}
\end{equation}
In order to have  a unified analysis, we also define $\Theta=0$  for the Crouzeix-Raviart element. Note that these auxiliary functions depend on the element $T$. We omit the dependence for simplicity of  notations.

\begin{lemma}\label{lem_jiest}
On each interface element $T\in\mathcal{T}^\Gamma_h$, these functions $\Upsilon_1(x)$, $\Upsilon_2(x)$, $\Psi(x)$ and $\Theta$ defined in (\ref{def_Upsilon1})-(\ref{def_Theta2})  exist and satisfy
\begin{equation}\label{est_uppsi}
\begin{aligned}
&|\Upsilon_i^s|^2_{H^m(T)}\leq Ch_T^{4-2m},~\quad m=0,1,2,~i=1,2,~s=+,-,\\
&|\Psi^s|^2_{H^m(T)}\leq Ch_T^{2-2m},~|\Theta^s|^2_{H^m(T)}\leq Ch_T^{6-2m},\quad m=0,1,2,~s=+,-,
\end{aligned}
\end{equation}
where the constant $C$  depends  only on $\beta_c^+$, $\beta_c^-$ and the shape regularity parameter $\varrho$.
\end{lemma}
\begin{proof}
We  construct $\Upsilon_i(x)$, $i=1,2$ as follows,
\begin{equation}\label{pro_Upsilon1}
\Upsilon_i=z_i-I_{h,T}^{\rm IFE}z_i, \quad z_i=\left\{
\begin{aligned}
&z_i^+\quad &&\mbox{ in } T_h^+,\\
&0\quad &&\mbox{ in } T_h^-,
\end{aligned}\right.,\quad i=1,2,
\end{equation}
where $z_1^+$ and $z_2^+$ are linear and satisfy 
\begin{equation}\label{pro_3.51}
\begin{aligned}
&z^+_1(D)=0, ~~\beta^+_c\nabla z^+_1 \cdot\textbf{n}_h=1, ~~\nabla z^+_1 \cdot\textbf{t}_h=0,\\
&z^+_2(D)=0, ~~\beta^+_c\nabla z^+_2 \cdot\textbf{n}_h=0, ~~\nabla z^+_2 \cdot\textbf{t}_h=1.
\end{aligned}
\end{equation}
It is easy to verify that $z^+_1$ and $z^+_2$ exist, and the constructed functions $\Upsilon_i$, $i=1,2$ satisfy (\ref{def_Upsilon1}) and (\ref{def_Upsilon2}).
From (\ref{pro_3.51}), we have
\begin{equation}\label{pro_Upsilon2}
\|z^+_i\|_{L^\infty(T)}\leq Ch_T,~ |\nabla z^+_i|\leq C, ~|z_i^+|_{W_\infty^2(T)}=0, ~ \|z_i\|_{L^\infty(T)}\leq Ch_T,\quad i=1,2.
\end{equation}
Since $I_{h,T}^{\rm IFE}z_i=\sum_{j}N_j(z_i)\phi_j$,  it follows from  (\ref{bounds_basis1}) that, for $m=0,1,2,~i=1,2,~s=+,-$,
\begin{equation*}
\left|\left(I_{h,T}^{\rm IFE}z_i\right)^s\right|_{W_\infty^m(T)}\leq \sum_{j\in\mathcal{I}}|N_j(z_i)||\phi^s_j|_{W_\infty^m(T)}\leq Ch_T^{-m}\sum_{j\in\mathcal{I}}|N_j(z_i)|.
\end{equation*}
Noticing 
$$|N_j(z_i)|\leq |e_j|^{-1}\int_{e_j}|z_i|ds\leq \|z_i\|_{L^\infty(T)}\leq Ch_T,$$
we have 
\begin{equation*}
\left|\left(I_{h,T}^{\rm IFE}z_i\right)^s\right|_{W_\infty^m(T)}\leq Ch_T^{1-m},\quad m=0,1,2,~i=1,2,~s=+,-,
\end{equation*}
which together with (\ref{pro_Upsilon1}) and (\ref{pro_Upsilon2}) implies 
\begin{equation*}
|\Upsilon_i^s|_{W_\infty^m(T)}\leq Ch_T^{1-m},~m=0,1,2,~i=1,2,~s=+,-.
\end{equation*}
Finally, the first estimate in (\ref{est_uppsi}) is obtained by
\begin{equation*}
|\Upsilon_i^s|^2_{H^m(T)}\leq |\Upsilon_i^s|^2_{W_\infty^m(T)} |T| \leq  Ch_T^{4-2m},\quad m=0,1,2,~i=1,2,~s=+,-.
\end{equation*}

Other  estimates in (\ref{est_uppsi}) can be proved similarly. We construct $\Psi(x)$ as
\begin{equation*}
\Psi=z-I_{h,T}^{\rm IFE}z,  \quad z=\left\{
\begin{aligned}
z^+=1\quad \mbox{ in } T_h^+,\\
0\quad \mbox{ in } T_h^-.
\end{aligned}\right.
\end{equation*}
It is easy to verify that the constructed function $\Psi$ satisfies (\ref{def_Psi1}) and (\ref{def_Psi2}).
Since
\begin{equation*}
\begin{aligned}
&\|z^+\|_{L^\infty(T)}=1,\quad |z^+|_{W_\infty^m(T)}=0,\quad \|z\|_{L^\infty(T)}=1, \quad m=1,2,\\
&|N_j(z)|\leq |e_j|^{-1}\int_{e_j}|z|ds\leq \|z\|_{L^\infty(T)}\leq 1,\quad j\in\mathcal{I},
\end{aligned}
\end{equation*}
we get
\begin{equation*}
\left|\left(I_{h,T}^{\rm IFE}z\right)^s\right|_{W_\infty^m(T)}\leq \sum_{j\in\mathcal{I}}|N_j(z)||\phi^s_j|_{W_\infty^m(T)}\leq Ch_T^{-m}\sum_{j\in\mathcal{I}}|N_j(z)|\leq Ch_T^{-m}
\end{equation*}
and
\begin{equation*}
|\Psi^s|_{W_\infty^m(T)}\leq Ch_T^{-m}, \quad s=+,-,~m=0,1,2,
\end{equation*}
which implies
\begin{equation*}
|\Psi^s|^2_{H^m(T)}\leq |\Psi^s|^2_{W_\infty^m(T)}|T| \leq Ch_T^{2-2m}, \quad s=+,-,~m=0,1,2.
\end{equation*}

For  the rotated-$Q_1$ element, we construct $\Theta(x)$ as
\begin{equation*}
\Theta=z-I_{h,T}^{\rm IFE}z,  \quad z=\left\{
\begin{aligned}
z^+=(x_1-m_1)^2-\kappa^2_T (x_2-m_2)^2\quad \mbox{ in } T_h^+,\\
0\quad \mbox{ in } T_h^-,
\end{aligned}\right.
\end{equation*}
where $(m_1,m_2)$ is the center of the rectangle $T$. 
Using (\ref{def_d}), we easily verify that the constructed function $\Theta$ satisfies (\ref{def_Theta1}) and (\ref{def_Theta2}). It follows that 
\begin{equation}
\begin{aligned}
&\|z^+\|_{L^\infty(T)}\leq Ch_T^2,~ |z^+|_{W^1_\infty(T)}\leq Ch_T, ~|z^+|_{W_\infty^2(T)}\leq C,~ \|z\|_{L^\infty(T)}\leq Ch_T^2,\\
&|N_j(z)|\leq |e_j|^{-1}\int_{e_j}|z|ds\leq \|z\|_{L^\infty(T)}\leq Ch_T^2,\quad j\in\mathcal{I},
\end{aligned}
\end{equation}
which implies
\begin{equation*}
\left|\left(I_{h,T}^{\rm IFE}z\right)^s\right|_{W_\infty^m(T)}\leq \sum_{j\in\mathcal{I}}|N_j(z)||\phi^s_j|_{W_\infty^m(T)}\leq Ch_T^{-m}\sum_{j\in\mathcal{I}}|N_j(z)|\leq Ch_T^{2-m}
\end{equation*}
and
\begin{equation}\label{new_theta}
\begin{aligned}
|\Theta^s|^2_{H^m(T)}&\leq |\Theta^s|^2_{W_\infty^m(T)}|T|\leq  |\Theta^s|^2_{W_\infty^m(T)}|T|\\
&\leq C\left( \left|\left(z\right)^s\right|^2_{W_\infty^m(T)}+  \left|\left(I_{h,T}^{\rm IFE}z\right)^s\right|^2_{W_\infty^m(T)}\right)|T|\\
&\leq Ch_T^{6-2m}, \quad s=+,-,~m=0,1,2.
 \end{aligned}
\end{equation}
For the Crouzeix-Raviart element, the above inequality (\ref{new_theta}) is also valid since we have defined $\Theta=0$ if $T$ is a triangle.
\end{proof}

\begin{lemma}\label{lema_fenjie}
On each interface element $T\in \mathcal{T}_h^\Gamma$,  for any $v\in \widetilde{H}^2(\Omega)$, the following identity holds:
\begin{equation}\label{fenjie1}
(I_{h} v_E^s)(x)-(I_{h}^{\rm IFE}v)^s(x)=a\Psi^s(x)+\sum_{i=1,2}b_i\Upsilon_i^s(x)+\sum_{i\in\mathcal{I}} g_i \phi^s_i(x)+t\Theta^s(x)~ \forall x\in T,~ s=+,-,
\end{equation}
with 
\begin{equation}\label{fenjie2}
\begin{aligned}
&a=[\![ I_{h} v_E^\pm ]\!](D),~b_1=[\![\beta_c^\pm \nabla (I_{h} v_E^\pm )\cdot \textbf{n}_h]\!](x_p),~b_2=[\![ \nabla (I_{h} v_E^\pm )\cdot \textbf{t}_h]\!](x_p),\\
&g_i=|e_i|^{-1}\left(\int_{e^+_i}(I_h v_E^+-v_E^+)ds+\int_{e^-_i}(I_h v_E^--v_E^-)ds\right),~i\in\mathcal{I},\\
&t=[\![d(I_hv_E^\pm)]\!],
\end{aligned}
\end{equation}
where  $e_i^s=e_i\cap\Omega^s$, $s=+,-$,  and $d(\cdot)$ is defined in (\ref{def_d}). It is easy to see that $g_i=0$ when $e_i^+=\emptyset$ or $e_i^+=e_i$.
\end{lemma}
\begin{proof}
For simplicity, define a  function $w_h$ such that  
\begin{equation}\label{wh}
w_h|_{T_h^s}=w_h^s~\mbox{ with }~w_h^s=I_{h} v_E^s-(I_{h}^{\rm IFE}v)^s, ~s=+,-.
\end{equation} 
It is obvious that $w_h^s\in V_h(T)$, $s=+,-$.
Define another function
\begin{equation}\label{newadd1}
\begin{aligned}
v_h:= [\![w_h^\pm]\!](D)\Psi(x)&+[\![\beta_c^\pm \nabla w_h^\pm\cdot\textbf{n}_h]\!](x_p)\Upsilon_1(x)\\
&+[\![ \nabla w_h^\pm\cdot\textbf{t}_h]\!](x_p)\Upsilon_2(x)+\sum_{i\in\mathcal{I}} N_i(w_h) \phi_i(x)+[\![d(w_h^\pm)]\!]\Theta(x).
\end{aligned}
\end{equation}
Next, we  prove $w_h=v_h$.
From (\ref{basis}), (\ref{shape1})-(\ref{shape2}) and Remark~\ref{remark_cc}, the IFE basis functions $\phi_i$, $i\in\mathcal{I}$ satisfy the following identities 
\begin{equation}\label{pro_lem_deco1}
\begin{aligned}
&[\![\phi_i^\pm]\!](D)=0, ~[\![\beta_c^\pm\nabla \phi_i^\pm\cdot\textbf{n}_h]\!](x_p) =0, ~[\![\nabla \phi_i^\pm\cdot\textbf{t}_h]\!](x_p) =0, \\
&[\![d(\phi_i^\pm)]\!]=0,~N_j(\phi_i)=\delta_{ij},~ i,j\in\mathcal{I}.
\end{aligned}
\end{equation} 
Combining (\ref{newadd1}), (\ref{pro_lem_deco1}) and (\ref{def_Upsilon1})-(\ref{def_Theta2}), we find
\begin{equation}\label{new_tnc2}
[\![\beta_c^\pm \nabla v_h^\pm\cdot\textbf{n}_h]\!](x_p)=[\![\beta_c^\pm \nabla w_h^\pm\cdot\textbf{n}_h]\!](x_p),~~N_i(v_h)=N_i(w_h),\quad i\in\mathcal{I}
\end{equation}
and
\begin{equation}\label{new_tnc}
[\![v_h^\pm]\!](D)=[\![w_h^\pm]\!](D),~~[\![ \nabla v_h^\pm\cdot\textbf{t}_h]\!](x_p)=[\![ \nabla w_h^\pm\cdot\textbf{t}_h]\!](x_p),~~[\![d(v_h^\pm)]\!]=[\![d(w_h^\pm)]\!].
\end{equation}
It follows from (\ref{new_tnc}) that $v_h-w_h$ is continuous across $\Gamma_h\cap T$, which together with (\ref{new_tnc2}) and (\ref{shape2.1})-(\ref{shape2}), implies  
$$v_h-w_h\in S_h(T)~\mbox{ and }~N_i(v_h-w_h)=0~ \forall i\in\mathcal{I}.$$
Using the unisolvence of IFE shape functions (see Lemma~\ref{lem_uniq}),  we know that the function $v_h-w_h$ is unique and  $v_h-w_h=0$ through a simple  verification. Thus, from (\ref{newadd1}), we have 
\begin{equation}\label{pro_lem_deco2}
\begin{aligned}
w_h=v_h= a\Psi(x)&+b_1\Upsilon_1(x)+b_2\Upsilon_2(x)+\sum_{i\in\mathcal{I}} g_i \phi_i(x)+t\Theta(x),
\end{aligned}
\end{equation}
where
\begin{equation*}
a=[\![w_h^\pm]\!](D),~b_1=[\![\beta_c^\pm \nabla w_h^\pm\cdot\textbf{n}_h]\!](x_p),~b_2=[\![ \nabla w_h^\pm\cdot\textbf{t}_h]\!](x_p),~g_i=N_i(w_h),~t=[\![d(I_hv_E^\pm)]\!].
\end{equation*}
Using the following facts from (\ref{shape2.1})-(\ref{shape2})
\begin{equation*}
[\![ I_{h}^{\rm IFE}v)^\pm ]\!](D)=0, ~[\![ \beta^\pm_c \nabla (I_{h}^{\rm IFE}v)^\pm\cdot\textbf{n}_h ]\!](x_p)=0,~ [\![\nabla (I_{h}^{\rm IFE}v)^\pm\cdot \textbf{t}_h ]\!](x_p)=0,
\end{equation*}
we further have 
\begin{equation*}
\begin{aligned}
&a=[\![w_h^\pm]\!](D)=[\![I_{h} v_E^\pm-(I_{h}^{\rm IFE}v)^\pm]\!](D)=[\![I_{h} v_E^\pm]\!](D)-[\![(I_{h}^{\rm IFE}v)^\pm]\!](D)=[\![I_{h} v_E^\pm]\!](D),\\
&b_1=[\![\beta_c^\pm \nabla w_h^\pm\cdot\textbf{n}_h]\!](x_p)=[\![\beta_c^\pm \nabla (I_{h} v_E^\pm-(I_{h}^{\rm IFE}v)^\pm)\cdot\textbf{n}_h]\!](x_p)=[\![\beta_c^\pm \nabla (I_{h} v_E^\pm)\cdot\textbf{n}_h]\!](x_p),\\
&b_2=[\![ \nabla w_h^\pm\cdot\textbf{t}_h]\!](x_p)=[\![ \nabla (I_{h} v_E^\pm-(I_{h}^{\rm IFE}v)^\pm)\cdot\textbf{t}_h]\!](x_p)=[\![ \nabla (I_{h} v_E^\pm)\cdot\textbf{t}_h]\!](x_p).
\end{aligned}
\end{equation*}
It remains to  consider $g_i$, $i\in \mathcal{I}$. Define 
\begin{equation*}
I_h^{BK}v:=
\left\{
\begin{aligned}
I_hv_E^+&&\mbox{ in } T_h^+,\\
I_hv_E^-&&\mbox{ in } T_h^- ,
\end{aligned}\right.
\end{equation*}
then  we know from (\ref{wh}) that $w_h=I_h^{BK}v-I_h^{\rm IFE}v$. Using the fact that $N_i(v-I_h^{\rm IFE}v)=0$ on interface elements from (\ref{local_IFE_inter}), we obtain
\begin{equation*}
\begin{aligned}
g_i&=N_i(w_h)=N_i(I_h^{BK}v-I_h^{\rm IFE}v)=N_i(I_h^{BK}v-v+v-I_h^{\rm IFE}v)\\
&=N_i(I_h^{BK}v-v)+N_i(v-I_h^{\rm IFE}v)=N_i(I_h^{BK}v-v)\\
&=|e_i|^{-1}\left(\int_{e^+_i}(I_h v_E^+-v_E^+)ds+\int_{e^-_i}(I_h v_E^--v_E^-)ds\right).
\end{aligned}
\end{equation*}
\end{proof}

\subsection{Proof of Theorem~\ref{chazhi_error}}\label{app_theo1}
\begin{proof}
On each interface element $T\in\mathcal{T}_h^\Gamma$, by  the triangle inequality, we have
\begin{equation}\label{pro_main1}
|v_E^s-(I_h^{\rm IFE}v)^s|_{H^m(T)}\leq |v_E^s-I_h v_E^s|_{H^m(T)}+|I_h v_E^s -(I_h^{\rm IFE}v)^s|_{H^m(T)},\quad s=+,-.
\end{equation}
The estimate of the first term is standard
\begin{equation}\label{pro_main2}
| v_E^s-I_hv_E^s|^2_{H^m(T)}\leq Ch_T^{4-2m}| v_E^s|^2_{H^2(T)},\quad m=0,1,2,~ s=+,-.
\end{equation}
For the second term on the right-hand side of (\ref{pro_main1}),  from Lemma~\ref{lema_fenjie}, Lemma~\ref{lem_jiest} and (\ref{bounds_basis1}),  we have
\begin{equation}\label{pro_main3}
\begin{aligned}
&|I_hv_E^s-(I_h^{\rm IFE}v)^s|^2_{H^m(T)}\\
&\leq C\left(a^2|\Psi^s|^2_{H^m(T)}+\sum_{i=1,2}b_i^2|\Upsilon_i^s|^2_{H^m(T)}+\sum_{i\in\mathcal{I}} g_i^2|\phi^s_i|^2_{H^m(T)}+t^2|\Theta^s|^2_{H^m(T)}\right)\\
&\leq Ch_T^{2-2m}\left(a^2+\sum_{i\in\mathcal{I}} g_i^2\right) +Ch_T^{4-2m}\sum_{i=1,2}b_i^2+Ch_T^{6-2m}t^2,\quad m=0,1,2.
\end{aligned}
\end{equation}
Here the constants $t$, $a$, $b_1$, $b_2$  and $g_i$, $i\in\mathcal{I}$ are defined in (\ref{fenjie2}). We now estimate these constants one by one. 
Firstly, (\ref{def_d}) implies 
\begin{equation}\label{est_t}
\begin{aligned}
t^2&=\left|[\![d(I_hv_E^\pm)]\!]\right|^2=\frac{1}{(4\kappa_T^4+4)|T|}|[\![I_hv_E^\pm]\!]|^2_{H^2(T)}\leq Ch_T^{-2}\sum_{s=+,-}|I_hv_E^s|^2_{H^2(T)}\\
&\leq Ch_T^{-2}\sum_{s=+,-}\left(|v_E^s|^2_{H^2(T)}+|v_E^s-I_hv_E^s|^2_{H^2(T)}\right)\leq Ch_T^{-2}\sum_{s=+,-}|v_E^s|^2_{H^2(T)}.
\end{aligned}
\end{equation}
Since $v\in \widetilde{H}^2(\Omega)=0$, we have  $[\![v_E^\pm]\!]^2(D)$, which leads to 
\begin{equation}\label{pro_a}
\begin{aligned}
a^2&=[\![ I_{h} v_E^\pm ]\!]^2(D)=[\![ I_hv_E^\pm-v^\pm_E]\!]^2(D)\leq \|[\![I_hv_E^\pm-v^\pm_E]\!]\|_{L^\infty(T)}^2\\
&\leq C\|I_hv_E^+-v_E^+\|^2_{L^\infty(T)}+C\|I_hv_E^--v_E^-\|^2_{L^\infty(T)}\leq Ch_T^2(|v_E^+|^2_{H^2(T)}+|v_E^-|^2_{H^2(T)}),
\end{aligned}
\end{equation}
where we used Sobolev’s inequality, scaling argument and the standard interpolation error estimate in the last inequality.


For the constant $b_1$, using the standard inverse inequality, we have
\begin{equation}\label{pro_b10}
\begin{aligned}
b_1^2&=[\![\beta_c^\pm \nabla (I_{h} v_E^\pm )\cdot \textbf{n}_h]\!]^2(x_p)\leq \|[\![\beta_c^\pm \nabla (I_{h} v_E^\pm )\cdot \textbf{n}_h]\!]\|^2_{L^\infty(T)}\leq Ch_T^{-2} \|[\![\beta_c^\pm \nabla (I_{h} v_E^\pm )\cdot \textbf{n}_h]\!]\|^2_{L^2(T)}\\
&\leq Ch_T^{-2} \|[\![\beta^\pm \nabla (I_{h} v_E^\pm )\cdot \textbf{n}_h]\!]\|^2_{L^2(T)}+Ch_T^{-2} \|[\![(\beta_c^\pm-\beta^\pm)\nabla (I_{h} v_E^\pm )\cdot \textbf{n}_h]\!]\|^2_{L^2(T)}.
\end{aligned}
\end{equation}
From (\ref{shape_condi}), the second term can be estimate as 
\begin{equation*}
\begin{aligned}
h_T^{-2}& \|[\![(\beta_c^\pm-\beta^\pm)\nabla (I_{h} v_E^\pm )\cdot \textbf{n}_h]\!]\|^2_{L^2(T)}\leq h_T^{-2}\sum_{s=+,-}\|\beta_c^s-\beta^s\|^2_{L^\infty(T)}\|\nabla (I_{h} v_E^s )\cdot \textbf{n}_h\|^2_{L^2(T)}\\
&\leq C\sum_{s=+,-}\left(\|\nabla (I_{h} v_E^s )\cdot \textbf{n}_h-\nabla v_E^s \cdot \textbf{n}_h\|^2_{L^2(T)}+\| \nabla v_E^s \cdot \textbf{n}_h\|^2_{L^2(T)}\right)\leq \sum_{s=+,-}|v_E|^2_{H^1(T)}.
\end{aligned}
\end{equation*}
For the first term on the right-hand side of (\ref{pro_b10}), using (\ref{error_nt}), we get
\begin{equation*}
\begin{aligned}
 &h_T^{-2}\|[\![\beta^\pm \nabla (I_{h} v_E^\pm )\cdot \textbf{n}_h]\!]\|^2_{L^2(T)} = h_T^{-2}\|[\![\beta^\pm \nabla (I_{h} v_E^\pm -v_E^\pm )\cdot \textbf{n}_h+\beta^\pm \nabla v_E^\pm \cdot (\textbf{n}_h-\textbf{n}+\textbf{n})]\!]\|^2_{L^2(T)}\\
 &\leq Ch_T^{-2}\left(\|[\![\beta^\pm \nabla (I_{h} v_E^\pm -v_E^\pm )\cdot \textbf{n}_h]\!]\|^2_{L^2(T)}+\|\textbf{n}-\textbf{n}_h\|^2_{L^\infty(T)}\|[\![\beta^\pm \nabla v_E^\pm ]\!]\|^2_{L^2(T)}+ \|[\![ \beta^\pm\nabla v_E^\pm\cdot \textbf{n}]\!]\|^2_{L^2(T)}\right)\\
 &\leq C\sum_{s=+,-}\left(|v_E^s|^2_{H^2(T)}+|v_E^s|^2_{H^1(T)}\right)+Ch_T^{-2}\|[\![ \beta^\pm\nabla v_E^\pm\cdot \textbf{n}]\!]\|^2_{L^2(T)}.
 \end{aligned}
\end{equation*}
Combining above three inequalities yields
\begin{equation}\label{pro_b1}
b_1^2\leq  C\sum_{s=+,-}\left(|v_E^s|^2_{H^2(T)}+|v_E^s|^2_{H^1(T)}\right)+Ch_T^{-2}\|[\![ \beta^\pm\nabla v_E^\pm\cdot \textbf{n}]\!]\|^2_{L^2(T)}.
\end{equation}
Analogously, 
\begin{equation}\label{pro_b2}
\begin{aligned}
&b_2^2=[\![ \nabla (I_{h} v_E^\pm )\cdot \textbf{t}_h]\!]^2(x_p)\leq \|[\![ \nabla (I_{h} v_E^\pm )\cdot \textbf{t}_h]\!]\|^2_{L^\infty(T)}\leq Ch_T^{-2}\|[\![ \nabla (I_{h} v_E^\pm )\cdot \textbf{t}_h]\!]\|^2_{L^2(T)}\\
&\leq Ch_T^{-2}\|[\![ \nabla (I_{h} v_E^\pm-v_E^\pm )\cdot \textbf{t}_h+ \nabla v_E^\pm \cdot (\textbf{t}_h-\textbf{t}+\textbf{t})]\!]\|^2_{L^2(T)}\\
&\leq Ch_T^{-2}\left(\|[\![ \nabla (I_{h} v_E^\pm-v_E^\pm )\cdot \textbf{t}_h]\!]\|^2_{L^2(T)}+\|\textbf{t}_h-\textbf{t}\|^2_{L^\infty(T)} \|[\![\nabla v_E^\pm]\!] \|^2_{L^2(T)}+\|[\![\nabla v_E^\pm\cdot\textbf{t}]\!]\|^2_{L^2(T)} \right)\\
&\leq C\sum_{s=+,-}\left(|v_E^s|^2_{H^2(T)}+|v_E^s|^2_{H^1(T)}\right)+Ch_T^{-2}\|[\![\nabla v_E^\pm\cdot\textbf{t}]\!]\|^2_{L^2(T)}.
\end{aligned}
\end{equation}
Finally, the constants $g_i$, $i\in\mathcal{I}$ in (\ref{fenjie2}) can be estimated by the  Cauchy-Schwarz inequality and the standard interpolation error estimate 
\begin{equation}\label{pro_di}
\begin{aligned}
g_i^2&= |e_i|^{-2}\left|\int_{e^+_i}(I_h v_E^+-v_E^+)ds+\int_{e^-_i}(I_h v_E^--v_E^-)ds\right|^2\leq Ch_T^{-1}\sum_{s=+,-}\|I_hv_E^s-v_E^s\|^2_{L^2(e_i)}\\
&\leq C\sum_{s=+,-}\left(h_T^{-2}\|I_hv_E^s-v_E^s\|^2_{L^2(T)}+\|\nabla(I_hv_E^s-v_E^s)\|^2_{L^2(T)}\right)\leq Ch_T^2\sum_{i=+,-}|v_E^s|^2_{H^2(T)}.
\end{aligned}
\end{equation}

We now combine (\ref{pro_main1})-(\ref{pro_a}), (\ref{pro_b1})-(\ref{pro_di}) to obtain the error estimate on interface elements
%
\begin{equation*}
\begin{aligned}
h_T^{2(m-1)}|v_E^s-(I_h^{\rm IFE}v)^s|^2_{H^m(T)}&\leq Ch_T^{2}\sum_{s=\pm}\|v_E^s\|^2_{H^2(T)}+C
\left(\|[\![ \beta^\pm\nabla v_E^\pm\cdot \textbf{n}]\!]\|^2_{L^2(T)}+\|[\![\nabla v_E^\pm\cdot\textbf{t}]\!]\|^2_{L^2(T)}\right).\end{aligned}
\end{equation*}
Summing up the estimate over all interface elements  $T\in\mathcal{T}_h^\Gamma$ and using Assumption~(\ref{assumption_delta}), we get
\begin{equation}\label{pro_main_ls}
\begin{aligned}
\sum_{T\in\mathcal{T}^\Gamma_h}h_T^{2(m-1)}|v_E^s&-(I_h^{\rm IFE}v)^s|^2_{W^m_2(T)}\leq Ch_\Gamma^{2}\|v_E^s\|^2_{H^2(\Omega)}+C
\left((\|[\![ \beta^\pm\nabla v_E^\pm\cdot \textbf{n}]\!]\|^2_{L^2(U(\Gamma,h_\Gamma))}\right.\\
&\qquad+\left.\|[\![\nabla v_E^\pm\cdot\textbf{t}]\!]\|^2_{L^2(U(\Gamma,h_\Gamma))}\right),~~s=+,-,~~m=0,1,2.
\end{aligned}
\end{equation}
Since $v\in \widetilde{H}^2(\Omega)$, we know from the definition (\ref{def_H2}) that $[\![\beta^\pm \nabla v_E^\pm \cdot \textbf{n}]\!](x)=0$ and $[\![ v_E^\pm ]\!](x)=0$  for all $x\in \Gamma$, which also implies $[\![\nabla v_E^\pm \cdot \textbf{t}]\!](x)=0$ on $\Gamma$.  Thus, by Lemma~\ref{strip},  and (\ref{ext_beta})-(\ref{new_vari_deri}),
\begin{equation*}
\begin{aligned}
&\left\|[\![\beta^\pm \nabla v_E^\pm\cdot \textbf{n}]\!]\right\|_{L^2(U(\Gamma,h_\Gamma))}^2\leq Ch_\Gamma^2\left|[\![\beta^\pm \nabla v_E^\pm\cdot \textbf{n}]\!]\right|_{H^1(U(\Gamma,h_\Gamma))}^2\leq Ch_\Gamma^2\sum_{s=+,-}\|v_E^s\|^2_{H^2(\Omega)},\\
&\left\|[\![\nabla v_E^\pm\cdot \textbf{t}]\!]\right\|_{L^2(U(\Gamma,h_\Gamma))}^2\leq Ch_\Gamma^2\left|[\![\nabla v_E^\pm\cdot \textbf{t}]\!]\right|_{H^1(U(\Gamma,h_\Gamma))}^2\leq Ch_\Gamma^2\sum_{s=+,-}\|v_E^s\|^2_{H^2(\Omega)}.
\end{aligned}
\end{equation*}
Substituting the above inequalities   into (\ref{pro_main_ls}) and using the extension result (\ref{extension}) we complete the proof.
\end{proof}

\subsection{Proof of Theorem~\ref{theorem_interpolation}}\label{app_theo2}
\begin{proof}
On each non-interface element $T\in\mathcal{T}_h^{non}$,  the following  estimate is standard
\begin{equation}\label{pro_int2a}
|v-I_h^{\rm IFE}v|^2_{H^m(T)}=|v-I_hv|^2_{H^m(T)}\leq Ch_T^{4-2m}|v|^2_{H^2(T)},\quad m=0,1.
\end{equation}
On each interface element $T\in\mathcal{T}_h^\Gamma$,  in view of the relations $T=T^+\cup T^-$ and $T^s=(T^s\cap T_h^+)\cup (T^s\cap T_h^-)$, $s=+,-$, we have
\begin{equation}\label{pro_int3a}
\begin{aligned}
|v-I_h^{\rm IFE}v|^2_{H^m(T)}&= \sum_{s=+,-}|v^s-(I_h^{\rm IFE}v)^s|^2_{H^m(T^s\cap T_h^s)}\\
&+|v^--(I_h^{\rm IFE}v)^+|^2_{H^m(T^-\cap T_h^+)}+|v^+-(I_h^{\rm IFE}v)^-|^2_{H^m(T^+ \cap T_h^-)}.
\end{aligned}
\end{equation}
By the triangle inequality, 
\begin{equation}\label{pro_int4}
\begin{aligned}
|v^--(I_h^{\rm IFE}v)^+|^2_{H^m(T^-\cap T_h^+)}\leq 2|v^--v_E^+|^2_{H^m(T^-\cap T_h^+)}+2|v_E^+-(I_h^{\rm IFE}v)^+|^2_{H^m(T^-\cap T_h^+)},\\
|v^+-(I_h^{\rm IFE}v)^-|^2_{H^m(T^+\cap T_h^-)}\leq 2|v^+-v_E^-|^2_{H^m(T^+\cap T_h^-)}+ 2|v_E^--(I_h^{\rm IFE}v)^-|^2_{H^m(T^+\cap T_h^-)}.
\end{aligned}
\end{equation}
Substituting (\ref{pro_int4}) into (\ref{pro_int3a}) and using the definition (\ref{tri}), we get
\begin{equation}\label{pro_int5}
|v-I_h^{\rm IFE}v|^2_{H^m(T)}\leq  C\sum_{s=+,-}|v^s-(I_h^{\rm IFE}v)^s|^2_{H^m(T)}+C|v_E^+-v_E^-|^2_{H^m(T^\triangle)},~m=0,1.
\end{equation}
It follows from Lemma~\ref{lem_h3} and the fact $[\![v_E^\pm]\!]=0$ on $\Gamma\cap T$ that 
\begin{equation*}
\begin{aligned}
&\|v_E^+-v_E^-\|^2_{L^2(T^\triangle)}\leq Ch_T^4\left|[\![v_E^\pm]\!] \right|^2_{H^1(T^\triangle)}\leq Ch_T^4\sum_{s=+,-}|v_E^s|^2_{H^1(T)},\\
&\|\nabla (v_E^+-v_E^-)\|^2_{L^2(T^\triangle)}\leq C\left(h_T^2\|\nabla [\![v_E^\pm]\!] \|^2_{L^2(\Gamma\cap T)}+h_T^4\left|[\![v_E^\pm]\!] \right|^2_{H^2(T^\triangle)}\right)\\
&\quad\qquad\qquad\qquad\leq Ch_T^2\sum_{s=+,-}\|\nabla v_E^s\|^2_{L^2(\Gamma\cap T)}  +Ch_T^4\sum_{s=+,-}|v_E^s|^2_{H^2(T)},
\end{aligned}
\end{equation*}
which implies
\begin{equation}\label{pro_int6}
|v_E^+-v_E^-|^2_{H^m(T^\triangle)}\leq Ch_T^{4-2m}\sum_{s=+,-}\left(|v_E^s|^2_{H^1(T)}+\|\nabla v_E^s\|^2_{L^2(\Gamma\cap T)}\right),~m=0,1.
\end{equation}
Combining (\ref{pro_int2a}), (\ref{pro_int5}) and (\ref{pro_int6}), we arrive at
\begin{equation*}
\begin{aligned}
&\sum_{T\in\mathcal{T}_h}|v-I_h^{\rm IFE}v|^2_{H^m(T)}\leq Ch^{4-2m}\sum_{T\in\mathcal{T}_h^{non}}|v|^2_{H^2(T)}+C\sum_{T\in\mathcal{T}_h^\Gamma}\sum_{s=+,-}|v^s-(I_h^{\rm IFE}v)^s|^2_{H^m(T)}\\
&\qquad\qquad\qquad+Ch^{4-2m}\sum_{T\in\mathcal{T}_h^\Gamma}\sum_{s=+,-}|v_E^s|^2_{H^1(T)}+Ch^{4-2m}\|\nabla v_E^s\|^2_{L^2(\Gamma)} \quad m=0,1,
\end{aligned}
\end{equation*}
which together with Theorem~\ref{chazhi_error}, the extension result (\ref{extension}) and the following global trace inequality
\begin{equation}\label{pro_lemmain8}
\sum_{s=+,-} \|\nabla v_E^s\|^2_{L^2(\Gamma)}\leq C(\|v_E^+\|^2_{H^2(\Omega^+)}+\|v_E^-\|^2_{H^2(\Omega^-)})
\end{equation}
yields the theorem.\end{proof}

\bibliographystyle{plain}
\bibliography{refer}

\end{document}